\documentclass[11pt,a4paper,reqno]{amsart}
\usepackage{amsaddr}
\usepackage{amsmath,amsfonts,amssymb,amsthm,mathrsfs}
\usepackage[utf8]{inputenc}
\usepackage[T1]{fontenc}
\usepackage{lmodern}
\usepackage{latexsym}
\usepackage{cancel}
\usepackage{microtype}
\usepackage{caption}
\usepackage{MnSymbol}
\usepackage{xcolor}
\usepackage{enumerate}
\usepackage[normalem]{ulem}
\usepackage{bbm}
\usepackage{geometry}
\usepackage{marginnote}

\usepackage[colorlinks=true,linkcolor=blue,citecolor=magenta]{hyperref}

\usepackage{graphics,tikz}

\newcommand{\BB}{{\mathcal  B}}

\newcommand{\EE}{{\mathcal  E}}
\newcommand{\FF}{{\mathcal  F}}

\newcommand{\HH}{{\mathcal  H}}
\newcommand{\TT}{{\mathcal  T}}

\newcommand{\MM}{{\mathcal  M}}

\newcommand{\PP}{{\mathcal  P}}

\newcommand{\BR}{{\mathbb R}}
\newcommand{\BX}{{\mathbb X}}

\def\Xint#1{\mathchoice
{\XXint\displaystyle\textstyle{#1}}%
{\XXint\textstyle\scriptstyle{#1}}%
{\XXint\scriptstyle\scriptscriptstyle{#1}}%
{\XXint\scriptscriptstyle\scriptscriptstyle{#1}}%
\!\int}
\def\XXint#1#2#3{{\setbox0=\hbox{$#1{#2#3}{\int}$ }
\vcenter{\hbox{$#2#3$ }}\kern-.6\wd0}}

\def\dashint{\Xint-}

\newtheorem{theorem}{\bf Theorem}[section]
\newtheorem{proposition}[theorem]{\bf Proposition}
\newtheorem{lemma}[theorem]{\bf Lemma}
\newtheorem{corollary}[theorem]{\bf Corollary}

\theoremstyle{definition}
\newtheorem{definition}[theorem]{Definition}
\newtheorem{example}[theorem]{\bf Example}

\newtheorem{remark}[theorem]{Remark}

\numberwithin{equation}{section}

\begin{document}

\title[Supersolutions to 
Schr\"odinger equations]{Location of   zeros of non-trivial positive supersolutions to 
Schr\"odinger equations}

\author{Tomasz Klimsiak}

\address{\small Institute of Mathematics, Polish Academy of Sciences,
ul. \'{S}niadeckich 8,   00-656 Warsaw, Poland, \and Faculty of
Mathematics and Computer Science, Nicolaus Copernicus University,
Chopina 12/18, 87-100 Toruń, Poland, \\e-mail: {\tt tomas@mat.umk.pl}}

\begin{abstract}
We study  Schr\"odinger operators on $L^2(E;m)$  of the form $-A+V$  with  singular
potentials $V$. We address the question posed by H. Brezis about the structure
of the set $\{u=0\}$ for non-negative supersolutions to $-Au+Vu=0$.
The class of operators $A$ we study in the paper  includes, in particular,
symmetric L\'evy type operators and symmetric diffusions in divergence form, 
with strictly positive Green functions. The class of potentials $V$
consists of positive  {\em smooth measures}, which contains, in particular, 
Coulomb potentials and harmonic potentials, as well as {\em generalized potentials},
i.e. positive Borel measures concentrated on $m$-negligible sets.

\end{abstract}
\maketitle



\section{Introduction}

Let $(E,\varrho)$ be a locally compact separable  metric space, $m$ be a
Radon measure on $E$ with full support and $A$ be a
self-adjoint operator on $L^2(E;m)$ that generates a Markov semigroup
of contractions  $(T_t)_{t\ge 0}$ on $L^2(E;m)$.  We assume
  that    the Green function $G$ for  the operator $A$ exists and  is strictly positive.

{\bf Formulation of the problem.} In the present paper, we answer the question posed by H. Brezis
(see \cite{Ancona,BB}): what is
the structure of the set  $\{u=0\}$, where $u$ satisfies
\begin{equation}
\label{eq1.1}
-Au+Vu\ge 0,\qquad u\ge 0,
\end{equation}
for a  locally integrable potential $V:E\to [0,\infty]$.
In fact in the paper we consider a wider class of potentials, allowing $V$ in  (\ref{eq1.1}) to be  a positive {\em smooth measure}  (see
Section \ref{sec2.1}), i.e. a Borel measure absolutely continuous with respect to the capacity $Cap_A$
naturally generated by the operator $A$, and such that $\int_E\eta\,dV<\infty$ for a strictly positive {\em quasi-continuous function}
$\eta:E\to\mathbb R$; so $V$ as a function need not be locally integrable and as a measure need not be  Radonian.
Thus, we encompass the class of so called {\em generalized Schr\"odinger operators}.

The classical strong maximum principle (Hopf \cite{Hopf}, 1927)  stands that if
$A=\Delta_{|D}$ (or $A$ is a diffusion with coefficients satisfying
some suitable assumptions) on a bounded domain $D\subset \BR^d$ and
 $V\in L^{p}(D;m)$ for some $p>d/2$, then  there are only two possibilities for   $u\in C^2_b(D)$
satisfying \eqref{eq1.1} pointwisely in $D$ :
either $\{u=0\}=\emptyset$ or $\{u=0\}=D$ (see e.g.  Littman \cite{Littman}, Stampacchia \cite{Stampacchia}, Gilbarg and Trudinger \cite{GT}).
In general, however,  for merely locally integrable potentials, we cannot expect  the
strong maximum principle to hold. For instance,
$\{u=0\}=\{0\}$ for the function $u(x)= |x|^2$, which satisfies
(\ref{eq1.1}) with $A=\Delta_{|B(0,1)}$ and $V(x)=2d|x|^{-2}$. We see that $V\in L^1(D;m)$ for $d\ge 3$ ($V\cdot m$ is a smooth measure for $d=2$).

The  problem of the structure of the set $\{u=0\}$ for $u$ satisfying \eqref{eq1.1},  
was studied  for  the Dirichlet Laplacian or uniformly elliptic diffusion operator by
Ancona \cite{Ancona}, B\'enilan and Brezis \cite{BB}, Brezis and
Ponce \cite{BP}, and recently by Orsina  and Ponce
\cite{OrsinaPonce}. The results obtained in
\cite{Ancona,BB,BP,OrsinaPonce} can be briefly summarized as follows.
In the paper by  Brezis and B\'enilan \cite[Appendix C]{BB} it is assumed that
$A=\Delta$ and  $V\in L^1_{loc}(\BR^d;m)$. It is
shown there  that if $u\in L^1_{loc}(\BR^d;m)$  satisfies (\ref{eq1.1}) a.e.,
then
 \[
\text{ boundedness of the set }\{u>0\}\mbox{ implies that } \{u=0\}=\BR^d\,\text{a.e.}
\]
Ancona \cite{Ancona} (see also Brezis and Ponce \cite{BP}) have considered a uniformly elliptic divergence form operator
$Au=\sum^{d}_{i,j=1}(a_{ij}u_{x_i})_{x_j}$ on  a bounded domain
$D\subset \BR^d$. In \cite{Ancona} it is proved that if a
quasi-continuous $u\in H^1(D)$ (or $u\in L^1(D;m)$ in case $a$ is
smooth) satisfies (\ref{eq1.1}) in the sense of measures, and is non-trivial ($m$-a.e.), then (for quasi-continuous version of $u$)
\begin{equation}
\label{eq.anc}
\qquad Cap_2(\{u=0\})=0.
\end{equation}
Here Cap$_2$ is the Newtonian capacity.
In the paper by Orsina and Ponce \cite{OrsinaPonce},  $A=\Delta_{|D}$ on  a bounded domain $D\subset \BR^d$ and  $V\in L^p(D;m)$ for some $p>1$. 
It is proved there that
if $u\in L^1(D;m)\cap L^1(D; V\cdot  m)$ satisfies (\ref{eq1.1}), and is non-trivial ($m$-a.e.), then
\begin{equation}
\label{eq1.2}
Cap_{W^{2,p}}\Big (x\in D:\limsup_{r\rightarrow 0^+}
\dashint_{B(x,r) }u(y)\,dy=0\Big)=0.
\end{equation}
Three conclusions can be drawn from the above results.  
They all concern    a "size" of the set $\{u=0\}$ - recall that if $Cap_A(B)=0$, then $\HH_G(B)<\infty$, where $\HH_G$
is the Hausdorff measure related to the Green function $G$ (see \cite{HN}), and if $Cap_{W^{2,p}}(B)=0$, then $\HH_{d-2p}(B)<\infty$,
where $\HH_{d-2p}$ is $(d-2p)$-dimensional Hausdorff measure.
For their formulation, if we are talking about a subtler measure than the Lebesgue measure, a  precise version of $u$ is needed, thus the existence of a regular  representative of $u$  must be a part of the assertion for the results describing the set $\{u=0\}$. Finally, the "size" of $\{u=0\}$ depends on the regularity of potential $V$.
 A companion problem  is the rigorous meaning of the  inequality \eqref{eq1.1}.  
 Besides the classical pointwise  formulation, applicable only to specific operators and regular $u$, 
 one may consider weak formulation for \eqref{eq1.1}:
 \begin{equation}
\label{eq5.1aint1}
-\int_E u A\eta\,dm+\int_E Vu\eta\,dm\ge 0,\quad \eta\in \mathcal C,
\end{equation}
 for a suitable   class $\mathcal C$ of test functions, and  one may understood \eqref{eq1.1}  in the sense of measures, with the assumption that
 $Au$ is a Borel measure. At this point it is worth mentioning that first results on strong maximum principle (for diffusion operators) with some non-classical  formulations of \eqref{eq1.1} are due to Calabi \cite{Calabi} and Littman \cite{Littman}.

{\bf Main results of the paper.}
In the present paper we study the structure of the set $\{u=0\}$ for a wide class of operators.
A complete novelty of  the present paper is the fact that we treat the question by H. Brezis for non-local operators
which  are now of great interest  both in pure and in applied mathematics (see e.g. \cite{BV} and references therein).
As we mentioned before, in the literature the attention has been focused on the problem of a "size" of the set $\{u=0\}$.
We go much further in this research, and this is the second novelty of the present paper, namely we indicate  a set $N_V$ - depending only on $A$ and $V$ - where all possible zeros of any non-trivial solution to \eqref{eq1.1} are located. As corollaries,  we get results on a "size"
of the set $\{u=0\}$. 
It appears that  the said set   $N_V$ admits the following form:
\[
N_V:= \{x\in E: \nexists\,\, \text{ finely-open } U_x, \text{with}\,\, x\in U_x,\,\,\text{such that  } \int_{U_x}G(x,y)V(y)\,m(dy)<\infty\}.
\]
This is an interesting object, which naturally appears in the context of Schr\"odinger equations with measure data (see \cite{K:NA1}),
and is well known in the probabilistic potential theory as it is the complement of the set of {\em permanent points} for $V$ (see e.g. \cite{BG}). 
It appears that independently
of the potential $V$ and operator $A$ we always have 
\begin{equation}
\label{eq.capp1}
Cap_A (N_V)=0.
\end{equation}
The  main result of the paper  (see Theorem \ref{th5.1}) stands as follows.

\vspace*{0.10cm}
\begin{center}
\begin{minipage}[c][2,85cm][t]{0,85\textwidth}
\textbf{Theorem 1.} 
Let  $u\in L^1(E;m)\cap L^1(E;V\cdot m)$ be a positive function satisfying 
\begin{equation}
\label{eq5.2int}
-\int_E u A\eta\,dm+\int_E Vu\eta\,dm\ge 0,\quad \eta\in \mathcal C,
\end{equation}
where $\mathcal C=\{\eta\in \mathfrak D(A):\eta\in \BB^+_b(E),\, A\eta\,\,\mbox{is bounded}\}$.
Then  there exists an $m$-version $\check u$
of $u$ which is {\em finely-continuous} on $E\setminus N_V$. Moreover, if $\check u(x)=0$ 
for some $x\in E\setminus N_V$, then $\check u\equiv0$.
\end{minipage}
\end{center}
Note  that the finely-continuous $m$-version $\check u$ of $u$ is given by the following formula:
\begin{equation}
\label{fcw101}
\check u(x)= \lim_{t\searrow 0}\int_Eu(y)p_t(x,y)\,m(dy)=\lim_{\alpha\to\infty}\alpha\int_Eu(y)G_\alpha(x,y)\,m(dy),\quad x\in E\setminus N_V,
\end{equation}
where $p_t(\cdot,\cdot)$ is the transition kernel of the semigroup $(T_t)$, and $G_\alpha(\cdot,\cdot)$ is its $\alpha$-Green's function.
The above result combined with \eqref{eq.capp1} yields a generalization of the result by Ancona  \eqref{eq.anc}.

From the above theorem, we conclude the second main result of the paper, which  stands that $N_V$ is exactly the set of 
  zeros of  all non-trivial positive functions satisfying \eqref{eq5.2int}.

\vspace*{0.10cm}
\begin{center}
\begin{minipage}[c][1,85cm][t]{0,85\textwidth}
\textbf{Theorem 2.} 
Let $\mathcal A :=\{u\in L^1(E;m)\cap L^1(E;V\cdot m):u\ge 0$, $u$ satisfies \mbox{\rm(\ref{eq5.2int})}, $u$ is finely-continuous and $u\nequiv 0$\}.
Then
\[
N_V= \bigcup_{u\in \mathcal A} \{u=0\}.
\]
 \end{minipage}
\end{center}

Although the  above theorems are the main results of the paper,  we put the major work  into proving
the following, interesting in its own, result  which  provides Feynman-Kac type representation for  solutions to \eqref{eq5.2int}.

\vspace*{0.10cm}
\begin{center}
\begin{minipage}[c][3,75cm][t]{0,85\textwidth}
\textbf{Theorem 3.} 
Assume that $u\in L^1(E;m)\cap L^1(E;V\cdot m)$ is a
positive function  that satisfies \eqref{eq5.2int}.
Then there exists an $m$-version $\check u$
of $u$ which is   finely-continuous on $E\setminus N_V$. Moreover,  for any $k\ge 0$ there exists a positive smooth measure $\beta_k$ such that
\begin{align*}
\check u_k(x)&=\int_\Omega \Big[e^{-\int_0^{t\wedge \tau_D(\omega)}V(\omega(r))\,dr}\check u_k(\omega(t\wedge \tau_D(\omega))\Big]\,dP_x(\omega)
\\&\qquad+\int_\Omega\Big[\int_0^{t\wedge \tau_D(\omega)}e^{-\int_0^rV(\omega(s))\,ds}\,dA^{\beta_k}_r(\omega)\Big]\,dP_x(\omega),
\end{align*}
for any open relatively compact set $D\subset E$, $x\in D$, and $t\ge 0$, where $u_k=u\wedge k$.
 \end{minipage}
\end{center}
Here $(P_x)_{x\in E}$ is a  family of Borel measures on the Skrochod  path space 
$\Omega\subset E^{[0,\infty)}$ consisting of c\`adl\`ag functions, so called Markov family associated with the semigroup
$(T_t)$, $A^{\beta_k}$ is a  {\em positive continuous additive functional} of $(P_x)$ associated with $\beta_k$ - in case $\beta_k$
 is a function $A^{\beta_k}_t(\omega)=\int_0^t\beta_k(\omega(r))\,dr$ -   and 
\[
\tau_D:\Omega\to [0,\infty],\quad \tau_D(\omega)=\inf\{t>0: \omega(t)\notin D\}
\]
(see introductory  Section \ref{sec2.2}).  

In Section \ref{sec6},
we prove our last main result of the paper concerning the strong maximum principle for operators $-A+V$ - we say that SMP holds 
for $-A+V$ if  for any {\em finely-continuous}  positive function $u\in L^1_{loc}(E;m)\cap L^1(E;V\cdot m)$ 
that satisfies \eqref{eq5.2int} we have the following  implication: if $u(x_0)=0$ for some $x_0\in E$, then $u(x)=0,\, x\in E$.
Section \ref{sec6} is the only one where we dispense with the assumption that $V\cdot m$
is a smooth measure. In this section the only requirement from $V:E\to [0,\infty]$ is being  Borel measurable.
By Theorem 1 if $V\cdot m$ is a smooth measure and $N_{V}=\emptyset$, then SMP holds.
However, $N_{V}=\emptyset$ already implies that $V\cdot m$ is  smooth. Thus implication $N_V=\emptyset\Rightarrow$ SMP is true
for arbitrary Borel measurable $V:E\to [0,\infty]$. We prove that if $N_V\neq\emptyset$, then there is only one possibility when SMP
still holds, namely if there is no non-trivial positive solution to \eqref{eq5.2int}.  The class of such potentials is reach. For example, in \cite{OP1}
it has been proven that for $A=\Delta$ and $V(x):=|x_1|^{-\gamma}$, with $\gamma\in [1,2)$ there is no non-trivial positive solution to \eqref{eq5.2int}.

\vspace*{0.10cm}
\begin{center}
\begin{minipage}[c][2,35cm][t]{0,85\textwidth}
\textbf{Theorem 4.} 
Let $V:E\to [0,\infty]$ be a Borel measurable function.
Then the SMP  holds for $-A+V$
if and only if 
\begin{enumerate}
\item[(1)] either $N_{V}=\emptyset$
\item[(2)] or $N_{V}\neq\emptyset$ and there is no non-trivial positive finely-continuous solution to \eqref{eq5.2int}.
\end{enumerate}
 \end{minipage}
\end{center}
The conclusion of the above theorem, in case $A=\Delta$ and $d=1$, follows from the recent paper by Bertsch, Smarrazzo and Tesei \cite{BST},
where the authors  went even a step further and gave a necessary and sufficient condition on $V$ guaranteeing that under $N_V\neq\emptyset$
there is no non-trivial positive solution to \eqref{eq5.2int} (see Remark \ref{rem.bst}). 

In Section \ref{sec7}, we generalize the result by Orsina and Ponce 
\eqref{eq1.2} (see also Section \ref{sec8}) by proving that for $V\in L^p(E;m)$, with $p>1$, and non-trivial positive $u\in L^1(E;m)\cap L^1(E;V\cdot m)$ that satisfies \eqref{eq5.2int} we have
\[
C_p(\{\check u =0\})=0,
\]
where $\check u$ is the finely-continuous $m$-version of $u$ on $N_V$, and $C_p$ is Riesz's capacity. 
Recall that for $A=\Delta$ we have $C_p\sim Cap_{W^{2,p}}$.

In Section \ref{sec8} (see Theorem \ref{th7.1})
we provide a result, especially important from the practical point of view, saying that in case $A$ is a L\'evy operator, i.e.
\begin{equation}
\label{eq.psi}
-A=\psi(-\Delta)
\end{equation}
for a Bernstein function $\psi$, then  Theorems 1--4 holds true  with $\mathcal C$ replaced by $C_c^\infty(\BR^d)$
in \eqref{eq5.2int}.  We also make some comments and provide some results concerning 
finely-continuous versions of Borel functions, especially for
the Laplacian and the fractional Laplacian.

{\bf Smooth measures.} In order to make the exposition of the main results of the present paper  
more clear,  we formulated them in the Introduction  for  potentials $V:E\to [0,\infty]$ such that measure $V\cdot m$ is {\em smooth}; equivalently
for potentials $V$ being locally {\em quasi-integrable}: for any compact $K\subset E$ and $\varepsilon>0$ there exists a closed set $F_\varepsilon\subset K$
such that $Cap_A(K\setminus F_\varepsilon)\le\varepsilon$ and 
\[
\int_{F_\varepsilon}V(y)\,m(dy)<\infty
\]
(we denote the class of such functions by $L^1_{q.loc}(E;m)$). Clearly, $L^1_{loc}(E;m)\subset L^1_{q.loc}(E;m)$. 
There are many interesting and very important in applications subclasses of $L^1_{q.loc}(E;m)$ which go beyond
the space $L^1_{loc}(E;m)$, such as, for example, the class of repulsive potentials 
\[
V(x)=\sum_{i=1}^l \frac{1}{\delta^{p_i}_{K_i}(x)},\quad x\in E.
\]
Here $K_i$ are compact sets satisfying $Cap_A(K_i)=0,\, i=1,\dots,l$ ($\delta_{K_i}(x):= \mbox{dist}(x,K_i)$),
and $p_i>0,\, i=1,\dots,l$. We see that $V$ explodes when approaches $K_i$.  
For the Dirichlet fractional  Laplacian $(\Delta^\alpha)_{|D}$, with a  bounded open  $D$ in $\BR^d$ ($d\ge 2$, $\alpha\in (0,1)$), 
we may take for $K_i$ any set of Hausdorff  dimension less than $d-2\alpha$ (see \cite[Theorem 5.1.9]{AH}). 

However, most of the results of the paper (Theorems 1--3) holds true with $V\cdot m$ replaced by 
a positive {\em smooth measure} $\nu$. The class  of {\em smooth measures}, denoted by $S_A$,
depends on the operator $A$, but, as we mentioned before, we  always have the inclusion $L^{1}_{loc}(E;m)\subset S_A$.
An interesting subclass of $S_A$,  of the great importance in applications,  consists of {\em generalized potentials}, 
i.e. positive measures which charge $m$-negligible   sets (see e.g. \cite{ABR}).
For example, for $A=(\Delta^\alpha)_{|D}$ any $\sigma$-finite positive Borel measure $\nu$ satisfying 
\[
\nu(dx)\ll \mathcal H_\lambda(dx),
\]
for some $\lambda\in (d-2\alpha,d)$, is a generalized potential (see \cite[Theorem 5.1.13]{AH}). 

It is known that $N_\nu=\emptyset$ if and only if  $\nu$ is {\em strictly smooth}, and this is so e.g.    when  $\nu$ is Green finite  or
$\nu$ is from Kato's class, see Section \ref{sec2.3}.
As a particular application of the main results of the paper, we have
that the classical strong maximum principle holds for the operator
$-\Delta_{|D}+\nu$ (with connected $D$) and $\nu(dy)=\gamma_M(dy)$ (where
$M$ is a compact $d-1$ dimensional manifold and $\gamma_M(dy)$ is the
surface measure on $M$), since in
this case $\nu$ is of Kato's class (see \cite{AM1}).

{\bf Dirichlet operators.}
 The class of operators considered  in the present paper includes the operators studied in \cite{Ancona, BB,BP,OrsinaPonce} 
 as well as other  fundamental  operators, for instance Laplacian with mixed boundary condition on a connected 
 open set (see \cite{Juan}), fractional Laplacian, Dirichlet fractional Laplacian, regional fractional Laplacian,  
 fractional Laplacian with mixed boundary condition (see \cite{BarriosMedina}) on arbitrary open set, L\'evy type operators of the
 form $\psi(-\Delta)$, where $\psi$ is a Bernstein function  (see \cite[Example 1.4.1]{FOT}, \cite{BL}), uniformly elliptic operators on 
 manifolds (see \cite[Example 5.7.2.]{FOT}), degenerate diffusion operators (see \cite[Exercise 3.1.1.]{FOT}), Laplacian on Sierpi\'nski's gasket (see \cite{FS}).
 The general structure of the operator $A$ is known in case $E\subset \BR^d$. 
 Due to Beurling-Deny decomposition  and the transformation rule for additive functionals (see \cite[Theorem 5.6.2]{FOT})
 \[
 (-Au,v)_{L^2(E;m)}=\int_E u_{x_i}v_{x_j}\,d\nu_{i,j}+\int_{E\times E\setminus \mathfrak d}(u(x)-u(y))(v(x)-v(y))J(dx,dy)
 \]
for any $u\in \mathfrak D(A),\, v\in\mathfrak D(\EE)$,
where $(\nu_{i,j})_{i,j=1,\dots d}$ is a positive definite matrix of smooth measures and  $J$ is a positive symmetric Borel measure
on $E\times E\setminus \mathfrak d$ ($\mathfrak d$ denotes the diagonal in $E\times E$).

{\bf Final comments.} 
After finishing this manuscript, we have learned about the results of \cite{OP}.
In this paper  Orsina and Ponce  studied the set of zeros of solutions to the equation 
\begin{equation}
\label{eq.op}
-\Delta u+Vu=f \text{  in  }  D,\quad u=0\text{  on  }\partial D,
\end{equation}
with $f\in L^\infty(D;m)$, where $V:E\to[0,\infty]$ is Borel measurable.  
The authors introduced in \cite{OP} a set $Z\subset D$,
called {\em universal zero-set},
\begin{equation}
\label{eq.zet}
Z:=\bigcap_{f\in L^\infty(D;m),f\ge 0}\{x\in D: \hat\omega_f(x):=\limsup_{r\rightarrow 0^+}
\dashint_{B(x,r) }w_f(y)\,m(dy)=0 \},
\end{equation}
where $\omega_f\in W^{1,2}_0(D)\cap L^\infty(D;m)\cap L^1(D;V\cdot m)$ is a unique solution to \eqref{eq.op}, and they proved 
that if $\hat w_f(x)=0$ for some $x\in D\setminus Z$, then $\hat w_f\equiv 0$ on a finely-connected component of $D\setminus Z$
containing $x$. From our results it follows easily that $Z=N_V$ provided $V$ is locally quasi-integrable. 
Thus, as a corollary, we obtain a simple characterization   of the set $Z$ by means of the Green function of $A$.
Observe that the result by Orsina and Ponce agrees with our results
since $D\setminus N_V$ is finely-connected due to \eqref{eq.capp1} (see Remark \ref{rem.fcom} and \cite[Corollary 1.2]{OP}).

\section{Notation and standing assumptions}
\label{sec2}

We denote  by $\BB(E)$ (resp. $\BB^+(E)$)  the set of all Borel (resp.
positive Borel) measurable functions on $E$. We say that a measure $\mu$ on $E$ is notrivial if $\mu(B)\neq0$ 
for some Borel set $B\subset E$. For $x\in E$ and $r>0$, $B(x,r):=\{y\in E:\varrho(x,y)<r\}$.

\subsection{Dirichlet forms and potential theory}
\label{sec2.1}

In the  paper, we assume that $(A,\mathfrak D(A))$ is a negative definite
self-adjoint
operator on $L^2(E;m)$  generating a strongly continuous Markov
semigroup of contractions $(T_t)_{t\ge 0}$ on $L^2(E;m)$. It is well known (see
\cite[Section 1]{FOT}) that there exists a unique symmetric
Dirichlet form $(\EE, \mathfrak D(\EE))$ on $L^2(E;m)$ such that
\[
\mathfrak  D(A)\subset \mathfrak D(\EE)\qquad \EE(u,v)=(-Au,v),\quad u\in \mathfrak D(A),\, v\in \mathfrak D(\EE).
\]
We assume that $(\EE,\mathfrak D(\EE))$ is transient and regular, i.e.
there exists a strictly positive bounded function $g$ on $E$ such
that
\[
\int_E|u|g\,dm\le \sqrt{\EE(u,u)},\quad u\in \mathfrak  D(\EE),
\]
and $\mathfrak D(\EE)\cap C_c(E)$ is dense in $C_c(E)$ in the uniform convergence topology,  and in $\mathfrak D(\EE)$  equipped with
the norm   generated  by the inner product  $\EE(\cdot,\cdot)+(\cdot,\cdot)_{L^2}$.

\begin{remark}
There is no loss of generality in assuming that $\EE$ is transient. Indeed, if $-A$ generates a Dirichlet form  $\EE$ which  is not transient, then for any $\alpha>0$ the operator $-A+\alpha$ generates the form $(\EE_\alpha,\mathfrak D(\EE))$, with $\EE_\alpha(\cdot,\cdot):=\EE(\cdot,\cdot)+\alpha(\cdot,\cdot)_{L^2(E;m)}$, which is transient. It is clear that if a positive $u$ satisfies $-Au+u\cdot\nu\ge 0$, then $-Au+\alpha u+u\cdot\nu\ge 0$. Moreover, if $\nu$ is smooth with respect to $\EE$, then it is  smooth with respect to $\EE_\alpha$. Therefore, we may apply the results of the paper to the operator $-A+\alpha$ perturbed by $\nu$.
\end{remark}

For an open set  $U\subset E$, we put
\[
Cap_A(U)=\inf\{\EE(u,u): u\ge\mathbf{1}_U\mbox{ $m$-a.e.},\,
u\in \mathfrak D(\EE)\},
\]
and then, for arbitrary $B\subset E$, we set $Cap_A(B)=\inf
Cap_A(U)$, where the infimum is taken over all open subsets $U$
of $E$ such that $B\subset U$. We say that a property holds
q.e. if it holds except for a set of capacity $Cap_A$ zero.

By $\MM(E)$ (resp. $\MM^+(E)$, $\MM_b(E)$) we denote the set of Borel  (resp. positive Borel, bounded Borel) measures
on $E$.
In the paper we adopt the following notation: for a $\mu\in\MM^+(E)$, and
$f\in \BB^+(E)$ we set
\[
\langle\mu, f\rangle=\int_E f\,d\mu.
\]
For $f$ and $\mu$ as above, we denote by  $f\cdot \mu$  the  Borel measure on $E$ such that
\[
\langle f\cdot\mu,\eta\rangle=\langle \mu,f\eta\rangle,\quad \eta\in \BB^+(E).
\]

We say that a function $u$ on $E$ is quasi-continuous if for every
$\varepsilon>0$ there exists a closed set $F_\varepsilon\subset E$
such that $Cap_A(E\setminus F_\varepsilon)\le\varepsilon$ and
$u_{|F_\varepsilon}$ is continuous. By \cite[Theorem 2.1.3]{FOT},
each function $u\in \mathfrak D(\EE)$ admits a quasi-continuous
$m$-version. In the sequel, for $u\in \mathfrak D(\EE)$, we denote by
$\tilde u$ its quasi-continuous $m$-version.

We say that  $\mu\in\MM(E)$  is {\em smooth}  if
\begin{enumerate}[(a)]
\item $\mu\ll Cap_A$,

\item $\langle|\mu|,\eta\rangle<\infty$ for some strictly positive quasi-continuous function $\eta$ on $E$.

\end{enumerate}
It is well known that every   $\mu\in\MM(E)$ admits the following unique decomposition
\[
\mu=\mu_d+\mu_c,
\]
where $\mu_d\ll Cap_A$ and $\mu_c\bot Cap_A$. In the
literature, $\mu_d$ is called the diffuse part of $\mu$, and
$\mu_c$ is called the concentrated part of $\mu$.

 We set
\[
\mathfrak D(\EE^\nu)= \mathfrak D(\EE)\cap L^2(E;\nu),\qquad
\EE^\nu(u,v)=\EE(u,v)+\langle \tilde u\cdot\nu,\tilde
v\rangle,\quad u,v\in \mathfrak D(\EE^\nu).
\]
By \cite[Theorem 4.6]{MR}, $(\EE^\nu,\mathfrak D(\EE^\nu))$ is a
quasi-regular symmetric Dirichlet form on $L^2(E;\nu)$. By
\cite[Corollary 2.10]{MR}, there exists a unique negative definite self-adjoint
operator $(A^\nu, \mathfrak D(A^\nu))$  such that $\mathfrak D(A^\nu)\subset
\mathfrak D(\EE^\nu)$ and
\[
\EE^\nu(u,v)=(-A^\nu u,v),\quad u\in \mathfrak D(A^\nu), v\in \mathfrak D(\EE^{\nu}).
\]
We put $-A+\nu:=-A^\nu$ and we denote
by $(T^\nu_t)_{t\ge 0}$  the strongly continuous
Markov semigroup of contractions on $L^2(E;m)$ generated by $-A^\nu$.

For an open set $D\subset E$, we denote by $(\EE^D,\mathfrak D(\EE^D))$ the part of $(\EE,\mathfrak D(\EE))$ on $D$, that is  a  symmetric form defined as
\[
\mathfrak D(\EE^D)=\{u\in \mathfrak D(\EE): \tilde u= 0\mbox{ q.e. on } E\setminus D\},\quad
\EE^D(u,v)= \EE(u,v),\quad u,v\in \mathfrak D(\EE^D).
\]
By \cite[Theorem 4.4.3]{FOT}, $(\EE^D, \mathfrak D(\EE^D))$ is again a regular symmetric transient Dirichlet form on $L^2(D;m)$. The operator generated by $(\EE^D,\mathfrak D(\EE^D))$ shall be
denoted by $A_{|D}$\,.

We denote  by $\Delta^{\alpha}$, $\alpha\in(0,1)$, the operator associated with the form
\[
\EE(u,v)=\int_{\BR^d}\hat u(x)\hat v(x) |x|^{2\alpha}\,dx,\quad \mathfrak D(\EE)=\{u\in L^2(\BR^d; dx): \int_{\BR^d}|\hat u|^2(x)|x|^{2\alpha}\,dx<\infty \},
\]
where $\hat u$ stands for the Fourier transform of $u$, 
and by $\Delta$ the usual Laplace operator, which can be viewed as the operator associated with the above form with $\alpha=1$.
It is well known that if $\alpha\in(0,1)$, then for $u\in C^\infty_c(\BR^d)$,
\[
\Delta^\alpha u(x)=c_{d,\alpha}\,P.V.\int_{\BR^d}\frac{u(y)-u(x)}{|x-y|^{d+2\alpha}}\,dy=\frac{c_{d,\alpha}}{2}\int_{\BR^d}\frac{u(x+y)+u(x-y)-2u(x)}{|y|^{d+2\alpha}}\,dy
\]
for some constant $c_{d,\alpha}>0$ (see, e.g., \cite{Kwasnicki}). The operator corresponding to the part of the above form on $D$ shall be denoted by $(\Delta^{\alpha})_{|D}$ if $\alpha\in(0,1)$, and by $\Delta_{|D}$ if $\alpha=1$.

\subsection{Probabilistic potential theory}
\label{sec2.2}
Let $\mathcal D$ be the set of functions $\omega:[0,\infty)\to E\cup \{\partial\}$ (where $\partial$ is a one-point compactification of $E$ in case $E$
is not compact, and an isolated point in case  $E$ is compact) such that
\begin{enumerate}[(1)]
\item $\omega$ is c\`adl\`ag, i.e. it is right continuous on $[0,\infty)$ and has left limits on $(0,\infty)$,
\item if $\omega(t)=\partial$ for some $t\ge 0$, then $\omega(s)=\partial,\, s\ge t$. 
\end{enumerate}
Let $d_S$ be the Skorochod metric on $\mathcal D$ (see  e.g. Section 12 of \cite{Bil}). With this metric $\mathcal D$ is a separable complete metric space.
Let $X:= Id_{\mathcal D}$ and
\[
X_t(\omega):= (X(\omega))(t),\quad t\ge 0,\, \omega\in \mathcal D.
\]
We see that $X_t$ is  a projection onto "$t$-coordinate".
From now on any function $u:E\to \BR\cup\{\pm\infty\}$  extends to $E\cup\{\partial\}$ by letting $u(\partial)=0$.
By \cite[Theorem 4.2.4]{FOT}, there exists a family $(P_x,\, x\in E\cup\{\partial\})$ of (Borel) probability measures on $\mathcal D$ such that
for any $f\in\BB_b(E)\cap L^2(E;m)$ and any $t>0$,
\begin{enumerate}[(1)]
\item $T_tf(x)=\int_{\mathcal D}f(X_t(\omega))\,P_x(d\omega)\,m$-a.e.,
\item $x\mapsto \int_{\mathcal D}f(X_t(\omega))\,P_x(d\omega)$ is quasi-continuous.
\end{enumerate}
As is customary, we denote 
\[
\mathbb E_x Y:= \int_{\mathcal D}Y(\omega)\,P_x(dw)
\]
for any  $Y\in\BB(\mathcal D)$. The operator $\mathbb E_x$ is called the {\em expectation} (with respect to $P_x$). Let $\mathcal P$ denote the set of all probability measures on $\BB(\BR)$.
For a given $\mu\in\PP$, we let 
\[
P_\mu(A):= \int_EP_x(A)\,\mu(dx),\quad A\in\BB(\mathcal D).
\]
Denote $\FF^0_t:=\sigma(X_s: s\le t)$, $\FF^0_\infty:= \sigma(X_s: s\ge 0)$.
Next, let $\FF^\mu_\infty$ be a completion of $\FF^0_\infty$ with respect to $P_\mu$,
and $\FF^\mu_t$ be a completion of $\FF^0_t$ in $\FF^\mu_\infty$ with respect to $P_\mu$.
We set
\[
\FF_t:=\bigcap_{\mu\in\PP}\FF^\mu_t,\quad \FF_\infty:=\bigcap_{\mu\in\PP}\FF_\infty^\mu.
\]
By \cite[Theorem 7.2.1]{FOT}
\[
\mathbb X=((X_t)_{t\ge0},\, (P_x)_{x\in E\cup \{\partial\}},\,
\mathbb{F}=(\FF_t)_{t\ge0},\, \zeta,\, (\theta_t)_{t\ge0})
\]
is  a Hunt process, i.e. admits some extra properties described in \cite[Section A.2]{FOT},
where 
\[
\zeta(\omega)=\inf\{t>0: X_t(\omega)=\partial\},\quad \theta_t:\mathcal D\to \mathcal D,\quad \theta_t(\omega)(s):=\omega(t+s).
\]
For $f\in\BB^+(E)$, we put
\[
P_t f(x)=\mathbb E_xf(X_t),\,\,t\ge0, \qquad R_{\alpha}f(x)=\mathbb E_x\int_0^\zeta
e^{-\alpha t}f(X_t)\,dt,\quad \alpha\ge0,\quad x\in E,
\]
and $R:=R_0$. We say that a property holds q.a.s. if it holds $P_x$-a.s. for q.e. $x\in E$, and we say that
it holds a.s. if it holds $P_x$-a.s. for every $x\in E$.

Any family $(Y_t)_{t\ge 0}$ of mappings $Y_t:\mathcal D\to \mathbb R $ that are $\mathcal F_\infty/\mathcal B(E)$
measurable for any $t\ge 0$ is called a {\em stochastic process}. We say that a stochastic process $(Y_t)$ is {\em $\mathbb F$-adapted}
if for each $t\ge 0$, $Y_t$ is $\mathcal F_t/\mathcal B(E)$ measurable. We say that a stochastic process $(Y_t)$ is c\`adl\`ag if
$P_x(Y_\cdot\in \mathcal D)=1,\, x\in E$.

Throughout the paper, we assume that $\mathbb X$ satisfies {\em absolute continuity condition}, which means that
$R_1f(x)=0,\, x\in E$ whenever $f\in L^2(E;m)\cap \BB^+(E)$ and $\int_Ef\,dm=0$.
Recall, see \cite[Theorem 4.2.4]{FOT}, that due to symmetry of $(\EE,\mathfrak D(\EE))$ absolute continuity of $\mathbb X$ 
implies that 
\begin{equation}
\label{eq.acs}
P_tf(x)=0,\,\, x\in E,\,\, t>0 \text{ whenever } f\in\BB^+(E) \text{ and }\int_Ef\,dm=0.
\end{equation}
We say that  $f\in\BB^+(E)$ is an $\alpha$-excessive function   if
\[
\sup_{t>0}e^{-\alpha t}P_t f(x)=f(x),\quad x\in E.
\]
It follows from the above definition and   the absolute continuity  of
$\mathbb X$, that   
\begin{equation}
\label{eq.acs1}
f,g \text{ are }\alpha\text{-excessive  and }f\le g\,\,  m\text{-a.e   imply  }   f(x)\le g(x),\, x\in E.
\end{equation}
We will use frequently this property
without special mentioning. In the sequel $0$-excessive functions will be called   simply excessive.

By \cite[Lemma 4.2.4]{FOT}, for any $\alpha\ge 0$, there exists a unique Borel function 
$G_\alpha: E\times E\to \BR^+\cup {\{+\infty\}}$ (called the $\alpha$-Green function) such that for all  $f\in\BB^+(E)$
\begin{equation}
\label{Green}
R_\alpha f(x)=\int_E G_\alpha(x,y)f(y)\,dy,\quad x\in E,
\end{equation}
and $G_\alpha(x,\cdot), G_\alpha(\cdot,y)$ are $\alpha$-excessive for any $x,y\in E$.
We let $G:=G_0$. For a given  $\mu\in\MM^+(E)$, we set
\[
R_\alpha \mu(x)=\int_E G_\alpha(x,y)\,\mu(dy),\quad x\in E.
\]
We also let  $R= R_0$. Observe that from the very definition of the Green function one readily concludes that 
\begin{equation}
\label{eq.acs0}
R_\alpha\mu \text{ is } \alpha\text{-excessive for any   }\mu\in\MM^+(E).
\end{equation}
By \cite[Corollary 1.3.6]{Oshima}, due to the assumption that $(\EE,\mathfrak D(\EE))$ is transient, there exists a strictly positive function $\eta \in \BB(E)$
such that
\begin{equation}
\label{eq.acs5}
R\eta(x)\le 1,\quad x\in E.
\end{equation}
From this, we  conclude that 
\begin{equation}
\label{eq.acs6}
\mu\in\MM^+_b(E) \text{  implies  }  R\mu<\infty\text{  q.e.}
\end{equation}
Indeed, it is enough to apply  \cite[Proposition II.3.5]{GetoorGlover} and the fact that $\langle R\mu,\eta\rangle=\langle\mu,R\eta\rangle\le \mu(E)<\infty$
for any strictly positive  function $\eta\in\BB(E)$ satisfying \eqref{eq.acs5}.

We say that $A\subset E$ is {\em nearly Borel} if there exist $B_1,
B_2\in\BB(E)$ such that $B_1\subset A\subset B_2$ and $Cap_A (B_2\setminus B_1)=0$.
The class of all
nearly Borel subsets of $E$ shall be denoted by $\BB^n(E)$. It is
clear that $\BB(E)\subset \BB^n(E)$. For  $B\in \BB^n(E)$, we set
\[
\sigma_B=\inf\{t>0: X_t\in B\},\qquad \tau_B=\sigma_{E\setminus
B}.
\]
By \cite[(A.2.7)]{FOT} for any $B\in\BB^n(E)$,
\begin{equation}
\label{eq.elch384}
\lim_{t\to 0^+} \tau_B\circ\theta_t+t=\tau_B.
\end{equation}
We say that a set $N\subset E$ is   {\em polar}  if there exists $B\in
\BB^n(E)$ such that $N\subset B$ and
\[
P_x(\sigma_B<\infty)=0,\quad x\in E.
\]
By \cite[Theorems 4.1.2,  4.2.1]{FOT}, $Cap_A(N)=0$ if
and only if $N$ is polar. 

Finally, observe that for any relatively compact open $D\subset E$
\begin{equation}
\label{eq.dpp19}
P_x(\tau_D<\infty)=1,\quad x\in E.
\end{equation}
Indeed, since $(\EE,\mathfrak D(\EE))$ is assumed to be regular,  $Cap_A(D)<\infty$. Thus, by \cite[Lemma 2.1.1,Lemma 2.2.6]{FOT}, there exists an excessive function $e_D$
that satisfies:
\[
e_D(x)=1,\, x\in D,\qquad e_D=R\mu \text{ for a positive smooth measure } \mu \text{ on } E.
\]
Therefore, applying \cite[proof (e), page 403]{Getoor} but with $Uh$, appearing there, replaced by $R\mu$ yields \eqref{eq.dpp19}.

Let $\TT$ be the topology generated by $\varrho$ (the metric on $E$). We denote by
$\TT_f$ the {\em fine topology} on $E$, that is  the smallest topology on $E$ for which all
excessive functions are continuous. By \cite[Section II.4]{BG}, $\TT\subset
\TT_f$ and $A$ is a finely-open set if and only if for every $x\in
A$ there exists $D\in\BB^n(E)$  such that $D\subset A$ and
\[
P_x(\tau_D>0)=1.
\]
In other words, starting from  $x\in A$,  the process $\mathbb X$ spends some
nonzero  time  in $A$ until it exits $A$. Observe that
each polar set is finely-closed. Therefore,  every non-empty finely-open set $V$
satisfies $Cap_A(V)>0$. By \cite[Theorem II.4.8]{BG},
\begin{equation}
\label{eq.acs2}
f\in \BB^n(E) \text{ is finely-continuous }\Leftrightarrow\,\, f(X) \text{   is right-continuous a.s. }
\end{equation}
Therefore, using absolute continuity condition,  one  concludes that 
\begin{equation}
\label{eq.acs3}
f,g \text{ are finely-continuous,  } f\le g\,\,  m\text{-a.e   imply  }   f(x)\le g(x),\, x\in E.
\end{equation}
By \cite[Theorem 4.6.1]{FOT}, 
\begin{equation}
\label{eq.acs4}
f \text{ is finely-continuous and finite q.e.  } \Rightarrow \,\, f\text{  is quasi-continuous}.
\end{equation}
If $w$ is an excessive function, then by \cite[Theorem III.5.7]{BG}, $w(X)$ is a c\`adl\`ag   $\mathbb F$-supermartingale under measure $P_x$ for every $x\in E$. In particular, $w$ is finely-continuous and hence quasi-continuous.

By \cite[Theorem A.2.10]{FOT}, a Hunt process
\[
\mathbb X^D=((X_t)_{t\ge0},\, (P^D_x)_{x\in D\cup \{\partial\}},\,
\mathbb{F}^D=(\FF^D_t)_{t\ge0},\,\zeta,\, (\theta_t)_{t\ge0})
\]
associated with the form $(\EE^D,\mathfrak D(\EE^D))$ satisfies
\begin{equation}
\label{eq.rrr1}
P^D_x(\zeta=\tau_D), \quad x\in D.
\end{equation}
Moreover,
\begin{equation}
\label{eq.rrr2}
P^D_t f(x):=\mathbb E^D_xf(X_t)=\mathbb E_x[f(X_t)\mathbf{1}_{\{t<\tau_D\}}],\quad x\in D.
\end{equation}

\begin{remark}
The notions of  excessive functions, harmonic functions, nearly Borel sets,
polar sets, fine
topology,  quasi-continuous functions, smooth measures,
$Cap_A$ introduced above  depend on the process $\mathbb X$ or the
associated Dirichlet form $(\EE, \mathfrak D(\EE))$. We omit this dependence
in our notation since in most of the present paper we use them
for a fixed process $\mathbb X$ and form $(\EE,\mathfrak D(\EE))$
associated with the operator $A$ in (\ref{eq1.1}). In the case where the
process or the form under consideration will change, we will write this explicitly.
\end{remark}

\subsection{Additive functionals and perturbation of
self-adjoint operators by  smooth measures}

\label{sec2.3}

\begin{definition}
\label{def.3.3.4af}
We say that an $\mathbb F$-adapted process $A=(A_t)_{t\ge0}$ is an  additive functional (AF)
of $\mathbb X$ if there exists a polar set $N$ and $\Lambda\in\FF_\infty$ such that
\begin{enumerate}[(a)]
\item$P_x(\Lambda)=1,\,\,\, x\in E\setminus N$,
\item $P_x(A_0=0)$, $x\in E\setminus N$, 
\item $\theta_t(\Lambda)\subset \Lambda,\, t>0$, and for every $\omega\in \Lambda$,
$A_{t+s}(\omega)=A_t(\omega)+A_s(\theta_s\omega)$, $s,t\ge 0$,
\item $A_t(\omega)<\infty$, $t<\zeta(\omega)$, $\omega\in\Lambda$, and
$A_t(\omega)=A_{\zeta(\omega)},\,\,\,  t\ge \zeta(\omega),\, \omega\in\Lambda$,
\item $[0,\infty)\ni t\mapsto A_t (\omega)$ is c\'adl\'ag for every $\omega\in\Lambda$.
\end{enumerate}
\end{definition}

The set $N$ in the above definition is called an {\em exceptional set} for $A$,
and $\Lambda$ is called a {\em defining set} for $A$.

\begin{definition}
\label{def.3.3.4paf}
We say that an $\mathbb F$-adapted process $A=(A_t)_{t\ge0}$ is a positive  additive functional (PAF)
of $\mathbb X$ if it is an additive functional with defining set $\Lambda\in \FF_\infty$ and exceptional set $N$, and 
\begin{enumerate}[(a)]
\item  $A_t(\omega)\ge 0$, $t\ge 0$, $\omega\in \Lambda$.
\end{enumerate}
\end{definition}

\begin{definition}
\label{def.3.3.4pcaf}
We say that an $\mathbb F$-adapted process $A=(A_t)_{t\ge0}$ is a positive continuous additive functional (PCAF)
of $\mathbb X$ if it is a positive  additive functional with defining set $\Lambda\in \FF_\infty$ and exceptional set $N$, and 
\begin{enumerate}[(a)]
\item $[0,\infty)\ni t\mapsto A_t (\omega)$ is continuous for every $\omega\in\Lambda$.
\end{enumerate}
\end{definition}

If $N=\emptyset $, then $A$ is called a strict PCAF of $\mathbb X$.

A Borel measure $\mu$ on $E$ is called  {\em strictly smooth}  if it is
smooth and  there exists an increasing sequence $\{B_n\}$ of Borel finely-open
subsets of $E$ such that $\bigcup_{n\ge 1} B_n=E$,  and
$R(\mathbf{1}_{B_n}\cdot|\mu|)$ is bounded for any $n\ge 1$.
By \cite[Theorem 5.1.7]{FOT}, there is a one-to-one correspondence between strict PCAFs 
and positive strictly smooth measures. By \cite[Theorem 2.2.4]{FOT}, for any positive smooth measure $\nu$,
there exists an increasing sequence $\{\nu_n\}$ of strictly smooth measures such that $\nu_n\nearrow \nu$, i.e.
$\nu_n(B)\nearrow \nu(B)$ for any $B\in\BB(E)$.

For  given $\alpha\ge 0$, $f\in\BB^+(E)$, and non-negative $\mathbb F$-adapted  c\`adl\`ag process  $Y$,
we let 
\begin{equation}
\label{eq.dofcaf12}
\phi^{\alpha,f}_Y(x)=E_x\int_0^\zeta e^{-\alpha t}f(X_t)e^{-Y_t}\,dt,\quad x\in E.
\end{equation}
The following result has been  proven in \cite{K:NA1}.  
Its close formulations  can be found in many  publications, however it should be emphasized that  the assertion of the  result below  provides a point-by-point analysis of the exceptional sets (e.g.  (iv) is well known but with  "q.e." in place of  "$x\in E_\nu$";  properties  (i) and (v) together, identify $N_\nu$ as a minimal exceptional set for an PCAF $A$ associated with $\nu$).

\begin{proposition}
\label{prop4.1pw}
Let $\nu$ be a positive smooth measure on $E$ and let
\begin{align*}
N_{\nu}=E\setminus E_{\nu},\qquad E_\nu=\Big\{x\in E:\exists& \, V_x  \mbox{ - finely-open neighborhood
of $x$}\nonumber \\
&\mbox{such that } \int_{V_x} G(x,y)\,\nu(dy)<\infty\Big\}.
\end{align*}
Let $\{\nu_n\}$ be a sequence of positive strictly smooth measures
such that $\nu_n\nearrow \nu$, and for $n\ge1$ let $A^n$ be a
strict PCAF of $\mathbb X$ in the Revuz correspondence with
$\nu_n$. Then,
\begin{enumerate}[\rm(i)]
\item The process
$A_t:=\sup_{n\ge 1} A^n_t$, $t\ge 0$, is a PCAF of $\BX$ with
the exceptional set $N_\nu $.
\item $\phi^{\alpha,f}_{A}$
is finely-continuous for any $\alpha>0$ and $f\in\BB^+_b(E)$.
\item If $f\in\BB^+(E)$ and $Rf$ is finite, then $\phi^{0,f}_{A}$
is finely-continuous.

\item For all $x\in E_\nu$ and $f\in\BB^+(E)$,
\begin{equation}
\label{eq4.1pw}
\mathbb E_x\int_0^\zeta f(X_t)\,dA_t=\int_EG(x,y)f(y)\,\nu(dy),
\end{equation}

\item For every $x\in N_\nu$, $P_x(A_t=\infty,\, t> 0)=1$.
\end{enumerate}
\end{proposition}

From now on, for a given smooth measure $\nu$, we denote by $A^\nu$ the PCAF of $\mathbb X$ with exceptional set $N_\nu$ constructed in Proposition \ref{prop4.1pw}. The one-to-one correspondence between PCAFs of $\mathbb X$ and positive smooth measures, expressed in our case by (\ref{eq4.1pw}), is called the Revuz duality.

In what follows we adopt the convention that for any $\mathbb F$-adapted positive process $Y$ and positive smooth measure $\nu$,
\begin{equation}
\label{eq.agreem}
\int_0^\zeta Y_r\,dA^\nu_r:=\lim_{n\rightarrow \infty}\int_0^\zeta Y_r\,dA^{\nu_n}_r\quad \mbox{a.s.},
\end{equation}
where $\{\nu_n\}$ is as in Proposition \ref{prop4.1pw}.

In the sequel, to emphasize  the dependence of the set $E_\nu$ on the operator $A$, we sometimes write $E_\nu(A)$, $N_\nu(A)$
instead of $E_\nu, N_\nu$. Observe that for any open set $D\subset E$ and positive smooth measure $\nu$,
\begin{equation}
\label{eq2.ier}
N_\nu(A_{|D})=N_\nu(A)\cap D.
\end{equation}
Indeed, from  (\ref{eq.rrr1}) and  (\ref{eq.rrr2}) it follows that  $A^\nu_{\cdot\wedge \tau_D}$ is a PCAF of $\mathbb X^D$ in Revuz duality with $\nu_{\lfloor D}$. Since, by Proposition \ref{prop4.1pw}, $A^\nu_{t}<\infty$, $t<\zeta$,  or $A^\nu_t=\infty$, $t>0$, we easily get  (\ref{eq2.ier}).

Thanks to the notion of PCAF of $\mathbb X$ one can give a
beautiful  probabilistic interpretation of the semigroup
$(T^\nu_t)_{t\ge 0}$ generated by the operator $-A+\nu$ (see
Section \ref{sec2.1}). It can be viewed as a generalization  of
the famous Feynman-Kac formula.
By \cite[Theorem A.2.11]{FOT}, there exists a Hunt process
\[
\mathbb X^\nu=((X_t)_{t\ge 0},\, (P^\nu_x)_{x\in E\cup \{\partial\}},\,\mathbb F^\nu=\{\FF^\nu_t,\, t\ge 0\},\,\zeta,\, (\theta_t)_{t\ge 0})
\]
associated with the form $(\EE^\nu,\mathfrak D(\EE^\nu))$ in the sense that  for every $f\in \BB(E)\cap L^2(E;m)$,
\begin{equation}
\label{eq4.s23}
T^\nu_tf(x)=\mathbb E^\nu_xf(X_t)\quad m\mbox{-a.e. }x\in E.
\end{equation}
We set
\[
P^\nu_tf(x)= \mathbb E^\nu_xf(X_t),\qquad R^\nu_\alpha f(x)= \mathbb E^\nu_x\int_0^\zeta e^{-\alpha t} f(X_t)\,dt,\quad x\in E,\quad \alpha\ge0,
\]
where $\mathbb E^{\nu}_x$ stands for the expectation with respect to $P^{\nu}_x$.
We put $R^\nu:=R^\nu_0$.
By \cite[Section 6.1]{FOT},
\begin{equation}
\label{eq4.s21}
P^\nu_tf(x)= \mathbb E_xe^{-A^\nu_t}f(X_t),\qquad R^\nu_\alpha f(x)=\mathbb E_x\int_0^\zeta e^{-\alpha t} e^{-A^\nu_t}f(X_t)\,dt,\quad x\in E_\nu.
\end{equation}
By \cite[Exercise 6.1.1]{FOT}, $-A+\nu$  possesses the Green function $G^\nu$ on  $E_\nu\times E_\nu$ and
\begin{equation}
\label{eq4.1}
G^\nu(x,y)+\int_{E_\nu}G(x,z)G^\nu(z,y)\,\nu(dz)=G(x,y),\quad (x,y)\in E_\nu\times E_\nu.
\end{equation}
From this identity and symmetry of $G$ and $G^\nu$, we infer that for any  $\mu\in\MM^+(E)$,
\begin{equation}
\label{eq4.2}
R^\nu\mu+R((R^\nu\mu)\cdot\nu)=R^\nu\mu+R^\nu(R(\mu_{\lfloor E_\nu})\cdot\nu)= R(\mu_{\lfloor E_\nu})\quad\mbox{in}\,\,\, E_\nu.
\end{equation}

\section{Irreducibility and Feynman-Kac formula}

In this section, we recall the notion of irreducibility of Markov semigroups $(T_t)_{t\ge 0}$, which in some sense (see Section \ref{Sec6})
is equivalent to obeying by $-A$ the strong maximum principle. We close the section with a simple  proposition
which suggests  how by the Feynman-Kac representation for  a function $u:E\to\mathbb R$ one can deduce the structure of the set $\{u=0\}$.

We say that an  $m$-measurable set $C\subset E$ is $(T_t)_{t\ge 0}$-invariant if for every
$f\in\BB^+(E)$, $T_t(\mathbf{1}_Cf)=\mathbf{1}_CT_tf,\, t\ge 0,\, m$-a.e.
A semigroup $(T_t)_{t\ge 0}$ is called irreducible if any invariant set $C$ satisfies $m(C)=0$ or $m(E\setminus C)=0$. It is known (see, e.g.,  \cite[Proposition 2.10]{BCR}) that an $m$-measurable set $C\subset E$ is $(T_t)_{t\ge 0}$-invariant  if and only if  there exists an excessive function $u_1$ (or $u_2$) such that $C=\{u_1=0\}$ (or $C=\{u_2<\infty\}$) $m$-a.e.
By \cite[Theorem 1.6.1]{FOT}, an $m$-measurable set $C\subset E$ is $(T_t)_{t\ge 0}$-invariant if and only if  for every $u\in \mathfrak D(\EE)$, $\mathbf{1}_C u\in \mathfrak D(\EE)$.

Let us also note here, that for any $\alpha$-excessive function $w$
the set $F:=\{w=0\}$ is finely-open (it is obviously  finely-closed as $w$ is finely-continuous).
Indeed, let  $V_n:=\{w>1/n\}$. Then, since $(e^{-\alpha t}w(X_t))_{t\ge 0}$ is a c\`adl\`ag supermartingale,  we have
\[
\mathbb E_x e^{-\alpha\sigma_{V_n}}w(X_{\sigma_{V_n}})\le w(x),\quad x\in E.
\] 
Note that for $x\in F$, the right-hand side of the above inequality equals zero. On the other hand, since $w(X)$
is c\`adl\`ag, we have that $e^{-\alpha\sigma_{V_n}}w(X_{\sigma_{V_n}})\ge e^{-\alpha\sigma_{V_n}}1/n>0$ provided  $\sigma_{V_n}<\infty$.
Therefore, for any $x\in F$, $P_x(\sigma_{V_n}<\infty)=0$. Letting $n\to \infty$, we get  $P_x(\tau_F<\infty)=0,\, x\in F$.
In other words, $F$ is finely-open.

\begin{lemma}
\label{lm4.1}
A symmetric Markov $C_0$-semigroup $(T_t)_{t\ge 0}$ on $L^2(E;m)$ satisfying the absolute continuity condition is irreducible if and only if
its Green function $G$ is strictly positive on $E\times E$.
\end{lemma}
\begin{proof}
Sufficiency is clear. By \cite[Exercise 4.7.1]{FOT}, for every finely-open $V\subset E$ such that $Cap_A(V)>0$, we have
\[
\int_V G_\alpha(x,y)\,m(dy)=\mathbb E_x\int_0^\zeta e^{-\alpha  t}\mathbf{1}_V(X_t)\,dt>0,\quad x\in E,\, \alpha\ge 0.
\]
Since $F_\alpha(x):=\{y\in E: G_\alpha(x,y)=0\}$ is finely-open (see  the comment preceding the lemma), we get by the above equation  that $Cap_A(F_\alpha(x))=0$ for any $\alpha\ge 0,\, x\in E$. Consequently, for every $x\in E$, $G_\alpha(x,\cdot)>0$ $m$-a.e., and so, by symmetry of $G_\alpha$,  $G_\alpha(\cdot,y)>0$ $m$-a.e. for every $y\in E$.
From  this and \eqref{eq4.1} applied to $\nu=m$, we infer that 
\[
G(x,y)\ge \int_EG(x,z)G_1(z,y)\,m(dz)>0,\quad x,y\in E.
\]
\end{proof}

\begin{remark}
\label{rem3.fc1}
Since it is known that the Green function for $(\Delta^\alpha)_{|D}$ is strictly positive
(see e.g. \cite{Grzywny}), it follows  from the above lemma that
for any open set $D\subset \BR^d$ and $\alpha\in (0,1)$, the semigroup generated by  $(\Delta^\alpha)_{|D}$ is irreducible.
Notice that this is not true for the classical Dirichlet Laplacian, because by \cite{Ramaswamy}, $\Delta_{|D}$ is irreducible if and only if $D$ is finely-connected.
\end{remark}

\begin{remark}
Let $(\EE,\mathfrak D(\EE))$ be a regular symmetric Dirichlet form on $L^2(E;m)$.  If $E$ is finely-connected, then $(T_t)_{t\ge 0}$ is irreducible (see \cite{Ramaswamy}).
If $\EE$ is local (i.e. $\EE(u,v)=0$ for any $u,v\in \mathfrak D(\EE)$ such supp$[u]\cap$supp$[v]=\emptyset$), then $(T_t)_{t\ge 0}$ is irreducible if and only if $E$ is finely-connected (see \cite{Ramaswamy}).
\end{remark}

\begin{lemma}
\label{lm4.2}
Let $\nu$ be a non-trivial positive smooth measure on $E$ and $(T_t)_{t\ge 0}$ be irreducible.
Then for every $x\in E$,
\[
P_x(\exists_{t>0}\,\, A^\nu_t>0)>0.
\]
\end{lemma}
\begin{proof}
By Proposition \ref{prop4.1pw}(iv), the assertion of the lemma holds  true for $x\in N_\nu$. If $x\in E_{\nu}$, then by (\ref{eq4.1pw}),
\[
\mathbb E_x\int_0^\zeta\,dA^\nu_r=\int_EG(x,y)\,\nu(dy),\quad x\in E_\nu.
\]
By Lemma \ref{lm4.1}, $G(x,y)>0,\, x,y\in E$. From this, the fact that $\nu$ is non-trivial and the above equality, we conclude that the result also  holds  true for $x\in E_{\nu}$.
\end{proof}

\begin{proposition}
\label{prop4.1}
Let $(T_t)_{t\ge 0}$ be irreducible and $u$ be a positive function on $E$. If there exist positive smooth measures $\mu,\nu$ such that $\mu$ is non-trivial and
\[
u(x)\ge \mathbb E_x\int_0^\zeta e^{-A^\nu_r}\,dA^\mu_r,\quad x\in E,
\]
then $\{u=0\}\subset N_\nu$.
\end{proposition}
\begin{proof}
Suppose that $u(x)=0$ for some $x\in E_\nu$.
Then  $\mathbb E_x\int_0^\zeta e^{-A^\nu_r}\,dA^\mu_r=0$, which
implies, in particular, that $x\in E_\mu$ (cf. \eqref{eq.agreem}). On the other hand, by
Lemma \ref{lm4.2}, $P_x(\exists_{t<\zeta}\,\, A^\mu_t>0)>0$.
Therefore $P_x(\exists_{t<\zeta}\,\, A^\nu_t=\infty)>0$. So, $x\in
N_\nu$, a contradiction.
\end{proof}

\section{Feynman-Kac representation for supersolutions} \label{Sec5}

In this section, we show that any   function $u$
satisfying (\ref{eq1.1}) has a finely-continuous version on
$E_\nu$, which admits a Feynman-Kac representation. This result
plays a pivotal role   in the proof of the main result of the paper. 

We start with providing  rigorous  meaning to  (\ref{eq1.1}). Set
\[
\mathcal U_b:=\{\eta\in \mathfrak D(A):\eta\in \BB^+_b(E),\, A\eta\,\,\mbox{is
bounded}\}.
\]

\begin{lemma}
\label{lm5.1}
The set $\mathcal U_b$ is dense in  $L^{2,+}(E;m)$ equipped with the standard norm.
\end{lemma}
\begin{proof}
Let $f\in L^{2,+}(E;m)$. By \cite[Theorem 2.4]{Pazy} and the fact that $(T_t)_{t\ge 0}$ is Markov,  $\frac 1t\int_0^tT_s(f\wedge k)\,ds\in \mathcal U_b,\, t>0,\, k>0$.  By the contractivity and the strong continuity of $(T_t)_{t\ge 0}$ we get the result.
\end{proof}

For positive   $u\in L^1(E;m)\cap L^1(E;\nu)$  we let
\[
I_u[\eta]:= \langle u,-A\eta\rangle+\langle u\cdot\nu,\tilde \eta\rangle ,\quad \eta\in \mathcal U_b.
\]

 \begin{definition}
\label{def.3.3.4maf1}
We say that  $\mathcal C \subset \mathcal U_b$ is a set of test functions for $A$ (we occasionally write $\mathcal C(A)$ to underline the operator) if 
 there exists $\mathcal C_1\subset \mathcal U_b$ that satisfies 
 \begin{enumerate}
 \item[{\rm (a)}] $\mathcal C\subset\mathcal C_1$,
 \item[{\rm (b)}]  $\mathcal C_1$ is dense in   $L^{2,+}(E;m)$, and satisfies  $R_\alpha \mathcal C_1\subset \mathcal C_1,\, \alpha>0$,
 \item[{\rm (c)}] for any positive $u\in L^1(E;m)\cap L^1(E;\nu)$, that is assumed additionally to be quasi-continuous in case $\nu$
 is not absolutely continuous with respect to $m$,  we have 
 \[
I_u[\eta]\ge 0,\quad \eta\in \mathcal C\qquad \Rightarrow \qquad  I_u[\eta]\ge 0,\quad \eta\in \mathcal C_1.
 \]
 \end{enumerate}
 \end{definition}
 
 \begin{definition}
\label{def.3.3.4maf}
We say that an $\mathbb F$-adapted process $A=(A_t)_{t\ge0}$ is a (local)  {\em martingale additive functional} (MAF)
of $\mathbb X$ if it is an additive functional with defining set $\Lambda\in \FF_\infty$ and exceptional set $N$, and 
 $A$ is a (local) martingale under measure $P_x$ for any $x\in E\setminus N$.
\end{definition}

\begin{theorem}
\label{prop5.1}  Let $\mathcal C$ be a set of test functions for $A$. Assume that $u\in L^1(E;m)\cap L^1(E;\nu)$ is a
positive function such that
\begin{equation}
\label{eq5.1a}
\langle u,-A\eta\rangle+\langle u\cdot\nu,\tilde \eta\rangle\ge 0,\quad \eta\in \mathcal C.
\end{equation}
 Suppose that  either $\nu\ll m$ or $u$ is quasi-continuous.
Then there exists an $m$-version $\check u$
of $u$ which is   finely-continuous on $E_\nu$. Moreover,  for every $k\ge 0$ there exists a positive smooth measure $\beta_k$ such that
\begin{equation}
\label{eq5.7a}
\check u_k(x)=\mathbb E_xe^{-A^\nu_{t\wedge \tau_D}}\check u_k(X_{t\wedge \tau_D})+\mathbb E_x\int_0^{t\wedge \tau_D}e^{-A^\nu_r}\,dA^{\beta_k}_r,\quad t\ge 0,
\end{equation}
for every open relatively compact set $D\subset E$ and every $x\in D$, where $u_k=u\wedge k$.
\end{theorem}
\begin{proof}
The proof shall be divided into five steps.\\
{\bf Step 1.} We  shall show that $\alpha R_\alpha(u+R(u\cdot\nu))\le u+R(u\cdot\nu)$ $m$-a.e. for every $\alpha>0$.
Let $\eta\in \mathcal C_1$ ($\mathcal C_1$ is an extension of $\mathcal C$ according to Definition \ref{def.3.3.4maf1}). By \cite[Theorem 2.4]{Pazy}
\[
\alpha R_\alpha\eta-\eta= A R_\alpha\eta.
\]
It follows from this and (\ref{eq5.1a}) that
\begin{align*}
\langle u, \alpha R_\alpha\eta-\eta\rangle +\langle R(u\cdot\nu),\alpha R_\alpha\eta-\eta\rangle&=\langle u, AR_\alpha\eta\rangle+\langle R(u\cdot\nu),\alpha R_\alpha \eta-\eta\rangle\\&\le \langle u\cdot\nu,R_\alpha\eta\rangle+\langle R(u\cdot\nu),\alpha R_\alpha\eta-\eta\rangle=0,
\end{align*}
where the last equality being a consequence of the resolvent   identity.
Since $\eta$ was an arbitrary function from $\mathcal C_1$, and  $\mathcal C_1$ is dense in $L^{2,+}(E;m)$, we get  the desired property.

{\bf Step 2.} We show that $u$ has an $m$-version $\check u$ which is  finely-continuous on $E_\nu$.
By \cite[Proposition 2.4]{BCR}, $w:=u+R(u\cdot \nu)$ possesses an $m$-version that  is excessive, and we let $\tilde w$ denote this version. 
By the construction, $\tilde w=\lim_{t\searrow 0}P_tw$.
Set $\bar u(x):=\limsup_{t\searrow 0}P_tu(x),\, x\in E$.
Observe that 
\[
\tilde w=\bar u+R(u\cdot\nu).
\]
From this (cf. \eqref{eq.acs0}, \eqref{eq.acs4}), in particular, we deduce that $\bar u$ is quasi continuous. 
Consequently, if $u$ is assumed to be quasi-continuous, then $u=\bar u$ q.e. (see \cite[Theorem 2.1.2]{FOT}) and, as a result,
$\bar u+R(\bar u\cdot\nu)=\bar u+R(u\cdot\nu)$, so $\bar u+R(\bar u\cdot\nu)$ is excessive. 
On the other hand, if $\nu$ is assumed to be absolutely continuous with respect to $m$, then 
clearly $\bar u+R(\bar u\cdot\nu)=\bar u+R(u\cdot\nu)$, as $\bar u$ is an $m$-version of $u$,
thus $\bar u+R(\bar u\cdot\nu)$ is excessive again. By \cite[Theorem III.5.7]{BG} $[\bar u+R(\bar u\cdot\nu)](X)$
is a c\`adl\`ag supermartingale under measure $P_x$ for q.e. $x\in E$.
By \cite[Theorem 3.18]{CJPS} there exists a PCAF $A$ of $\mathbb X$
and a local MAF of $\mathbb X$ such that for q.e. $x\in E$
\begin{equation}
\label{eqmm0}
\bar u(X_t)=\bar u(X_0)+\int_0^t\bar u(X_r)\,dA^\nu_r-A_t+M_t,\quad t\ge 0,\quad P_x\mbox{-a.s.}
\end{equation}
Applying integration by parts formula to  the product $e^{-A^\nu_t}\bar u(X_t)$ yields
\begin{equation}
\label{eqmm1}
e^{-A^\nu_t}\bar u(X_t)=\bar u(X_0)-\int_0^te^{-A^\nu_r}\,dA_r+\int_0^t e^{-A^\nu_r} \,dM_r,\quad t\ge 0,\quad P_x\mbox{-a.s.}
\end{equation}
Let $(\tau_k)$ be a non-decreasing sequence of stopping times such that $\tau_k\to \infty$, $\mathbb E_x\int_0^{\tau_k}e^{-A^\nu_r}\,dA_r\le k$ q.e., and $\Big(\int_0^{t\wedge \tau_k} e^{-A^\nu_r} \,dM_r\Big)_{t\ge 0}$ is a martingale under $P_x$ q.e.
Then, by \eqref{eqmm1}
\begin{equation}
\label{eq.mmqq}
\mathbb E_x e^{-A^\nu_{t\wedge\tau_k}}\bar u(X_{t\wedge\tau_k})\le \bar u(x)\quad\mbox{q.e.}
\end{equation}
From this and Fatou's lemma, we deduce that $P^\nu_t\bar u\le \bar u$ q.e. for any $t\ge 0$ (see also \eqref{eq4.s21}).
Thus, by \cite[Proposition 2.4]{BCR},  $\check u:= \lim_{t\searrow 0} P^\nu_t\bar u$ is excessive on $E_\nu$ with respect to $(P^\nu_t)$.
We put  $\check u(x)=0$ for $x\in N_\nu$. 
Since $\check u$ is an excessive function with respect to $(P^\nu_t)_{t\ge 0}$ on $E_\nu$, it is finely-continuous on $E_\nu$ with respect to $(P^\nu_t)_{t\ge 0}$.
This is equivalent to the fact that the process $t\mapsto e^{-A^\nu_t}\check u(X_t)$ is right-continuous under the measure $P_x$ for $x\in E_\nu$
(see \eqref{eq.acs2} and \eqref{eq4.s21}).
Since for any $x\in E_\nu$ we have  $e^{-A^\nu_t}>0,\, t\ge 0,\, P_x$-a.s.,  we see  that $\check u(X)$ is right-continuous
under the measure $P_x$ for $x\in E_\nu$. In other words, $\check u$ is finely-continuous on $E_\nu$ (see \eqref{eq.acs2}). 
In particular, $\check u$ is quasi-continuous (see \eqref{eq.acs4}). Therefore, with  the aid of  \cite[Theorem 2.1.2]{FOT}, $\check u=\bar u$ q.e. as $\check u, \bar u$ are $m$-versions of $u$, and 
both $\check u$ and $\bar u$ are quasi-continuous. 
As a result, we get, by \eqref{eqmm0}, that for q.e. $x\in E$
\begin{equation}
\label{eqmm0p}
\check u(X_t)=\check u(X_0)+\int_0^t\check u(X_r)\,dA^\nu_r-A_t+M_t,\quad t\ge 0,\quad P_x\mbox{-a.s.}
\end{equation}

{\bf Step 3.}
We  show that (\ref{eq5.7a}) holds q.e. 
Write $\check u_k= \check u\wedge k$. By the Tanaka-Meyer formula  (see, e.g., \cite[IV.Theorem 70]{Protter})
applied to \eqref{eqmm0p} we have
\begin{align}
\label{eq5.3}
\nonumber \check u_k(X_t)&=\check u_k(X_0) +\int_0^t\mathbf{1}_{\{\check u\le k\}}(X_r)\check u(X_r)\,dA^{\nu}_r\\
&\quad-\int_0^t\mathbf{1}_{\{\check u\le k\}}(X_r)\,dA_r
-C^k_t+\int_0^t\mathbf{1}_{\{\check u\le k\}}(X_{r-})\,dM_r,\quad t\ge 0\,\,\,\mbox{q.a.s.,}
\end{align}
where $C^k$ is an increasing c\`adl\`ag process with $C^k_0=0$.
Let $\{\tau_n\}$  be a fundamental sequence  (for the definition
see, e.g., \cite[Section I.6]{Protter}) for the local martingale
$\int_0^\cdot\mathbf{1}_{\{\check u\le k\}}(X_{r-})\,dM_r$. By
(\ref{eq5.3})
\begin{align*}
\check u_k(x)+\mathbb E_x\int_0^{\tau_n}\mathbf{1}_{\{\check u\le k\}}(X_r)\check u(X_r)\,dA^{\nu}_r=\mathbb E_xC^k_{\tau_n}+\mathbb E_x\check u_k(X_{\tau_n})+\mathbb E_x\int_0^{\tau_n}\mathbf{1}_{\{\check u\le k\}}(X_r)\,dA_r\quad\mbox{q.e.}
\end{align*}
Letting $n\rightarrow \infty$ we get  $\mathbb E_xC^k_\zeta\le k+R(u\cdot\nu)(x)$ q.e. Thus  $\mathbb E_xC^k_\zeta<\infty$  q.e. 
Let $C^{k,p}$ be the dual predictable projection of $C^k$ (see \cite[Section A.3]{FOT}). It exists q.e. since $\mathbb E_xC^k_\zeta<\infty$  q.e. By (\ref{eq5.3}), $C^k$ is a positive additive functional, so by \cite[Theorem A.3.16]{FOT}, $C^{k,p}$ is also a positive additive functional.
By the definition of a Hunt process, $\mathbb X$ is quasi-left continuous.  Therefore, by \cite[Proposition 2, Proposition 4]{CW} every local $\mathbb F$-martingale has only totally inaccessible   jumps.  So, by (\ref{eqmm0p}), $\check u(X)$  has only totally inaccessible   jumps.
Consequently, by (\ref{eq5.3}), $C^{k,p}$ has only totally inaccessible   jumps too. However, $C^{k,p}$ is predictable.
 Therefore  $C^{k,p}$ is continuous. By the Revuz duality, there exists a unique positive smooth measure $\gamma_k$ such that
$C^{k,p}=A^{\gamma_k}$. By the definition of the dual predictable projection, there exists a uniformly integrable martingale $N$ such that $C^k=C^{k,p}+N$. We let  
\[
L^k_t=N_t+ \int_0^{t}\mathbf{1}_{\{\check u\le k\}}(X_{r-})\,dM_r,\quad t\ge 0\,\,\,\mbox{q.a.s.}
\]
Furthermore, by the Revuz duality there exists a positive smooth measure $\mu$ such that $A=A^\mu$.
Consequently, by (\ref{eq5.3})
\begin{align}
\label{eq5.4}
\nonumber \check u_k(X_t)&=\check u_k(X_0)+\int_0^t\mathbf{1}_{\{\check u\le k\}}(X_r)\check u(X_r)\,dA^{\nu}_r\\
&\quad-\int_0^t\mathbf{1}_{\{\check u\le k\}}(X_r)\,dA^{\mu}_r-A^{\gamma_k}_t+L^k_t,\quad t\ge 0\,\,\,\mbox{q.a.s.}
\end{align}
Now,  put $\beta_k=(\mathbf{1}_{\{\check u>k\}}\check u_k)\cdot\nu+\mathbf{1}_{\{\check u\le k\}}\cdot\mu+\gamma_k$. From (\ref{eq5.4}) we conclude that
\begin{equation}
\label{eq5.5}
\check u_k(X_t)=\check u_k(X_0)+\int_0^t\check u_k(X_r)\,dA^{\nu}_r-A^{\beta_k}_t+L^k_t,\quad t\ge 0\,\,\,\mbox{q.a.s.}
\end{equation}
Applying  the integration by parts formula  to the product $e^{-A^\nu_t}\check u_k(X_t)$ yields
\begin{equation*}
e^{-A^\nu_t}\check u_k(X_t)=\check u_k(X_0)-\int_0^te^{-A^\nu_r}\,dA^{\beta_k}_r+\int_0^te^{-A^\nu_r}\,dL^k_r,\quad t\ge 0\,\,\,\mbox{q.a.s.}
\end{equation*}
Let $\{\tau_n\}$  be a fundamental sequence for the local martingale $L^k$.
By the above equation for any stopping time $\alpha$ we have 
\begin{equation}
\label{eqmm5.7}
\check u_k(x)=\mathbb E_xe^{-A^\nu_{\alpha\wedge\tau_D\wedge \tau_n}}\check u_k(X_{\alpha\wedge\tau_D\wedge \tau_n})+\mathbb E_x\int_0^{\alpha\wedge\tau_D\wedge \tau_n}e^{-A^\nu_r}\,dA^{\beta_k}_r.
\end{equation}
Letting $n\to \infty$, we get  that for q.e. $x\in D$ and any stopping time $\alpha$,
\begin{equation}
\label{eq5.7}
\check u_k(x)=\mathbb E_xe^{-A^\nu_{\alpha\wedge\tau_D}}\check u_k(X_{\alpha\wedge\tau_D})+\mathbb E_x\int_0^{\alpha\wedge\tau_D}e^{-A^\nu_r}\,dA^{\beta_k}_r.
\end{equation}
We used here \eqref{eq.dpp19}.

{\bf Step 4.} We  show that $N_{\beta_k}\subset N_\nu$. First  observe that by (\ref{eq5.7}), $\mathbb E_x\int_0^\zeta
e^{-A^\nu_r}\,dA^{\beta_k}_r\le k$ for q.e. $x\in E$. Let
$\{\beta^k_n\}$  be a sequence of smooth measures with bounded
potentials such that $\beta^k_n\nearrow \beta_k$ (see \cite[Theorem 2.2.4]{FOT}). The function
 $w_n=\mathbb E_\cdot\int_0^\zeta
e^{-A^{\nu}_r}\,dA^{\beta^k_n}_r$ is finely-continuous.  Indeed,
by \cite[Lemma 5.1.5]{FOT},
\[
w_n(x)=\mathbb E_x\int_0^{\zeta}\,dA^{\beta^k_n}_r-\mathbb E_x\int_0^\zeta
w_n(X_r)\,dA^\nu_r,\quad x\in E.
\]
By \eqref{eq.acs0} both functions on the right-hand side of the above equation are
finely-continuous. Moreover,   $\mathbb E_\cdot\int_0^{\zeta}\,dA^{\beta^k_n}_r$ is
bounded, so $w_n$ is finely-continuous. Therefore $w_n(x)\le k$
for every $x\in E$ (see \eqref{eq.acs3}). 
From this inequality, definition of $w_n$, and Proposition \ref{prop4.1pw}(iv),  we easily deduce that
$N_{\beta_k}\subset N_\nu$.

{\bf Step 5.} Conclusion. Fix $t>0$, and put
\[
v(x)=\mathbb E_x\int_0^{t\wedge \tau_{D}}e^{-A^\nu_r}\,dA^{\beta_k}_r,
\quad w(x)=\mathbb E_xe^{-A^\nu_{t\wedge{\tau_{D}}}}\check u_k(X_{t\wedge\tau_{D}}),\quad x\in D.
\]
It is an elementary check that  if $s\searrow0$,
then $t\wedge( \tau_{D}\circ\theta_s)+s\searrow t\wedge\tau_{D}$ a.s. (cf. \eqref{eq.elch384}). For  $x\in E_\nu$, if
$s\searrow0$, then
\begin{align}
\nonumber\label{eq5.10} P_s(w)(x)&=P_s(\mathbb E_\cdot e^{-A^\nu_{t\wedge\tau_D}}\check
u_k(X_{t\wedge\tau_D}))(x)\\&
=\mathbb E_xe^{-(A^\nu_{t\wedge( \tau_{D}\circ\theta_s)+s}-A^\nu_s)}\check
u_k(X_{t\wedge(\tau_{D}\circ\theta_s)+s})\rightarrow w(x).
\end{align}
We have used here Markov property of $\mathbb X$, fine continuity of $\check u_k$ on $E_\nu$ and
continuity of $A^\nu$ under the measure $P_x$ for $x\in E_\nu$.
Next, observe that by the strong Markov property of $\mathbb X$,
\begin{align}
\label{eq5.11}
v(x)=R^\nu\beta_k(x)-\mathbb E^\nu_x\big((R^\nu\beta_k)(X_{t\wedge\tau_{D}})\big),\quad x\in E_\nu.
\end{align}
Since $N_{\beta_k}\subset N_\nu$, $A^{\beta_k}$ is a strict PCAF of $\mathbb X^\nu$. Hence, by \cite[Proposition II.4.2]{BG}, $R^\nu\beta_k$
is finely-continuous with respect to $(P^\nu_t)_{t\ge 0}$. In particular, by \eqref{eq.acs2} and \eqref{eq4.s21},
\[
\lim_{s\searrow 0}P_s(R^\nu\beta_k)(x)=\lim_{s\searrow 0}P^\nu_s(R^\nu\beta_k)(x)=R^\nu \beta_k(x),\quad x\in E_\nu.
\]
Since $R^\nu\beta_k$ is bounded one  shows,   as in the case of  $w$, that
\[
P_s(\mathbb E^\nu_\cdot R^\nu\beta_k(X_{t\wedge\tau_{D}}))(x)\rightarrow \mathbb E^\nu_xR^\nu\beta_k(X_{t\wedge\tau_{D}}),\quad x\in E_\nu,\quad  s\searrow 0.
\]
Thus $P_s v(x)\rightarrow v(x)$ for $x\in E_\nu$ as $s\searrow 0$. By \eqref{eq5.7} 
\[
\check u_k=w+v,\quad \mbox{q.e.}
\]
Therefore, by the absolute continuity condition for $\mathbb X$ (see \eqref{eq.acs}),
\[
P_s(\check u_k)(x)=P_s(w)(x)+P_s(v)(x),\quad x\in E_\nu.
\]
Letting $s\searrow 0$  we get the desired result.
\end{proof}

\begin{corollary}
\label{wn.red} Let $\mathcal C$ be a set of test functions for $A$. Let $u\in L^1(E;m)\cap L^1(E;\nu)$ be a
quasi-continuous positive  function such that
\mbox{\rm(\ref{eq5.1a})} is satisfied. Then
\[
\langle u\wedge k,-A\eta\rangle+\langle (\mathbf{1}_{\{u\le k\}}u)\cdot\nu,\eta\rangle\ge 0,\quad \eta\in \mathcal C.
\]
\end{corollary}
\begin{proof}
We adopt the notation of Theorem \ref{prop5.1} and its proof. Let   $\{D_n\}$ be an  increasing sequence
of relatively compact open subsets of $E$ satisfying $\bigcup_{n\ge 1} D_n=E$.
Substituting $\tau_{D_n}$ in place of $t$ in \eqref{eq5.5} (cf. \eqref{eq.dpp19}),
putting the expectation to both sides of \eqref{eq5.5}, and then  letting $n\to \infty$, we find that 
\[
u_k=h^{(k)}-R(u_k\cdot\nu)+R\beta_k\quad\mbox{a.e.,}
\]
where $h^{(k)}(x):=\lim_{n\to\infty}\mathbb E_x\check u_k(X_{\tau_{D_n}}),\, x\in E$. It is an elementary check that $h^{(k)}$
is an excessive function up to $m$-equivalence. Let $\eta\in \mathcal C$. Then, by the above equation,
\[
\langle u_k,-A\eta\rangle+\langle R((\mathbf1_{\{u\le k\}}u)\cdot \nu),-A\eta\rangle=\langle h^{(k)},-A\eta\rangle+\langle R\beta_k,-A\eta\rangle\ge 0.
\]
In the last inequality, we used the facts that $h^{(k)}, R\beta_k$ are excessive functions, and $\eta$ is positive.
Now, one easily concludes the result.
\end{proof}

\section{Location of zeros of non-trivial supersolutions}
\label{Sec6}

It appears that without additional assumption on the Green function $G^\nu$, the  set of positive function $u\in L^1(E;m)\cap L^1(E;\nu)$ satisfying (\ref{eq5.1a}) may be trivial.

\begin{proposition}
\label{prop5.9wd4} There exists a non-trivial quasi-continuous positive  function
$u\in L^1(E;m)\cap L^1(E;\nu)$  satisfying
\mbox{\rm(\ref{eq5.1a})} if and only if  $G^\nu(x_0,\cdot)\in
L^1(E;m)$ for some $x_0\in E_\nu$.
\end{proposition}
\begin{proof}
Suppose that $u\in L^1(E;m)\cap L^1(E;\nu)$ is the asserted function. 
By  the reasoning following (\ref{eq.mmqq}),
there exists a function $\check u$ that is  excessive  with respect to $(P^\nu_t)_{t\ge 0}$
and equals $u$ q.e.
Let $\mu$ be a bounded positive non-trivial measure on $E_\nu$.
Then $R^\nu\mu$ is  an excessive function with respect to
$(P^\nu_t)_{t\ge 0}$ which is finite a.e. (see \eqref{eq.acs0}, \eqref{eq.acs6}). Therefore $u\wedge
R^\nu\mu$ shares the same properties. By \cite[Proposition 3.9]{GetoorGlover},
there exists a positive Borel measure $\gamma$ on $E_\nu$ such
that $u\wedge R^\nu\mu=R^\nu\gamma$. Of course, $\gamma$ is
non-trivial, and   $R^\nu\gamma\in L^1(E;m)$. Hence 
\[
\langle \gamma,R^\nu 1\rangle=\langle R^\nu\gamma,1\rangle=\|R^\nu\gamma\|_{L^1(E;m)}<\infty.
\]
Thus, $R^\nu1<\infty$ $\gamma$-a.e. Since $\gamma$ is non-trivial, there exists $x_0\in E_\nu$ such that 
\[
(R^\nu1)(x_0)=\int_E G^\nu(x_0,y)\,m(dy)<\infty.
\]
Now suppose that there exists $x_0\in E_\nu$ such that $G^\nu(x_0,\cdot)\in L^1(E;m)$.
Then $u:= G^\nu(x_0,\cdot)\in L^1(E;m)\cap L^1(E;\nu)$ and  $u$ satisfies (\ref{eq5.1a}). Indeed, it is clear that $u$ is an excessive function
with respect to $(P^\nu_t)_{t\ge 0}$, so (\ref{eq5.1a}) is satisfied.  Furthermore, by (\ref{eq4.2}), $R^\nu\nu\le 1$, so $
\langle u,\nu\rangle=\langle R^\nu\delta_{x_0},\nu\rangle=\langle \delta_{x_0},R^\nu\nu\rangle\le 1$.
\end{proof}

In the reminder of this section, we assume that $G^\nu(x_0,\cdot)\in L^1(E;m)$ for some $x_0\in E_\nu$. This assumption is satisfied for instance  if $m(E)<\infty$ or $R1$ is bounded.

\begin{theorem}
\label{th5.1} Let $\mathcal C$ be a set of test functions for $A$.  Assume that   $(T_t)_{t\ge 0}$ is irreducible. Let  
$u\in L^1(E;m)\cap L^1(E;\nu)$ be a positive function satisfying 
\begin{equation}
\label{eq5.2}
\langle u,-A\eta\rangle+\langle u\cdot\nu,\tilde\eta\rangle\ge 0,\quad \eta\in \mathcal C.
\end{equation}
 Suppose that  either $\nu\ll m$ or $u$ is quasi-continuous.
Then 
there exists an $m$-version $\check u$
of $u$ which is finely-continuous on $E_\nu$. Moreover, if $\check u(x)=0$ for some $x\in E\setminus N_\nu$, then $\check u\equiv0$.
\end{theorem}
\begin{proof}
Let $\check u$ be the function constructed in the  proof of Theorem \ref{prop5.1}. It is  finely-continuous on $E_\nu$ and $\check u(x)=0$, $x\in N_\nu$.  By Theorem \ref{prop5.1}, for any $k\ge 0$ and  relatively compact open set $D\subset E$,
\begin{equation}
\label{eq5.7adcf}
\check u_k(x)=\mathbb E_xe^{-A^\nu_{t\wedge \tau_D}}\check u_k(X_{t\wedge \tau_D})+\mathbb E_x\int_0^{t\wedge \tau_D}e^{-A^\nu_r}\,dA^{\beta_k}_r,\quad t\ge 0,
\end{equation}
for every $x\in D$, where $\beta_k$ is a positive smooth measure and $\check u_k=\check u\wedge k$. From (\ref{eq5.7adcf}) it follows that
\begin{equation}
\label{eq5.8}
\check u_k(x)\ge \mathbb E_x\int_0^{\zeta}e^{-A^\nu_r}\,dA^{\beta_k}_r,\quad x\in E_\nu.
\end{equation}
Assume  that $\check u(x)=0$ for some $x\in E_\nu$.  By Proposition \ref{prop4.1} and (\ref{eq5.8}), $\beta_k= 0$. 
Therefore, letting $t\to \infty$ in  (\ref{eq5.7adcf}), we get
\begin{equation}
\label{eq5.9}
\check u_k(x)=\mathbb E_xe^{-A^\nu_{\tau_{D}}}\check u_k(X_{\tau_{D}}),\quad t\ge 0,\quad x\in D.
\end{equation}
Let $\{D_n\}$ be an increasing sequence of relatively compact open subsets of $E$ such that $\bigcup_{n\ge 1}D_n=E$.
Applying strong Markov property to  (\ref{eq5.9}) we find that for any $n\ge 1$, and any stopping time $\alpha$,
\[
e^{-A^\nu_\alpha}\check u_k(X_\alpha)\mathbf1_{\{\alpha<\tau_{D_n}\}}=\mathbb E_x(e^{-A^\nu_{\tau_D}}\check u_k(X_{\tau_{D_n}})\mathbf1_{\{\alpha<\tau_{D_n}\}}|\FF_\alpha)\quad\mbox{a.s.}
\]
From this and the choice of sequence $\{D_n\}$, we infer that  for any bounded stopping time $\alpha$, 
$\mathbb E_x e^{-A^\nu_\alpha}\check u_k(X_\alpha)=\check u(x),\, x\in E$. As a result,
$(\check u_k(X_t)e^{-A^\nu_t})_{t\ge 0}$ is a  martingale under the measure
$P_x$   for any  $x\in E$. Therefore, if $\check u(x)=0$ for some $x\in E_\nu$, then $e^{-A^\nu_t}\check u_k(X_t)=0,\, t\ge 0,\, P_x$-a.s. Since $x\in E_\nu$, we have that $e^{-A^\nu_t}>0,\, t<\zeta,\, P_x$-a.s. Therefore,  $\check u_k(X_t)=0,\, t\ge 0,\, P_x$-a.s. Consequently,
\[
0=\mathbb E_x\int_0^\zeta \check u_k (X_r)\,dr=\int_E G(x,y)\check u(y)\,dy.
\]
Since $(T_t)_{t\ge 0}$ is irreducible, $G(x,\cdot)$ is strictly positive (see Lemma \ref{lm4.1}). Thus, $\check u=0\,\,\, m$-a.e.
Since $\check u$ is finely-continuous on $E_\nu$, and $\mathbb X$ satisfies absolute continuity condition, we have $\check u=0$
on $E_\nu$  (see \eqref{eq.acs3}). By the very definition of  $\check u=0$, we have  $\check u(x)=0,\, x\in N_\nu$.
\end{proof}

\begin{example}
\label{examp.1}
Observe that if $\nu\ll m$ does not hold, then, in general,  the conclusion of Theorem \ref{th5.1}
does not hold without assumption that $u$ be quasi-continuous. Indeed, let
\[
w(x):=1-(2-|x|)^+\wedge 1,\quad x\in E:=(-2,2),\quad A:=\Delta,\quad \nu:=\delta_{\{-1\}}+\delta_{\{1\}}.
\]
One easily finds  that $-\Delta w=-\nu$, i.e. $\langle w,-\Delta\eta\rangle=-\int_E\eta\,d\nu,\, \eta\in C^\infty_c(E)$. Set
\[
u(x):=w(x),\quad x\in E\setminus\{-1,1\},\quad u(-1)=u(1):=1.
\]
Then, $-\Delta u+u\cdot\nu\ge 0$, i.e. $\langle u,-\Delta\eta\rangle+\int_Eu\eta\,d\nu\ge 0,\, \eta\in C^\infty_c(E)$.
Since the Green function for $-\Delta$ on $(-2,2)$ is bounded, we have $N_\nu=\emptyset.$
We see that $u(x)=0, x\in [-1,1]$ but $u\nequiv 0$ in $E_\nu=E$. This is so because $u$ is not continuous on $E$,
which in our case, i.e.  one dimension and $A=\Delta$, is equivalent to quasi-continuity of $u$.
\end{example}

\begin{corollary}
\label{cor5.1g} Let the assumptions of Theorem \ref{th5.1} hold with quasi-continuous $u$ (if $u$ is not assumed to be quasi-continuous in  Theorem \ref{th5.1}, then by its assertion we know that such version exists). If
$Cap_A(\{u=0\})>0$, then $u=0$ q.e.
\end{corollary}
\begin{proof}
Since $u$ is quasi-continuous, $\check u= u$ q.e. (see \cite[Theorem 2.1.2]{FOT}).  Hence
$Cap_A(\{\check u=0\})>0$. From this and the fact that $Cap_A(N_\nu)=0$ we conclude that there exists $x\in E_\nu$
such that $\check u(x)=0$. Therefore, by Theorem \ref{th5.1}, $\check
u=0$, so $u=0$ q.e.
\end{proof}

\begin{remark}
The above corollary, in case $A=\sum_{i,j}\frac{\partial}{\partial x_j}(a_{ij}\frac{\partial }{\partial x_i})$ is a symmetric uniformly elliptic operator  on a bounded domain, was proved by Ancona \cite[Theorem 8]{Ancona}.
In fact, Ancona needed some additional regularity assumptions on  $u$ or coefficients $a_{ij}$ (either $u\in H^1_{loc}(D)$ or $a_{i,j}$ are smooth and $Au$ is a measure).
Note, however, that  thanks to these additional conditions the author dispensed with the assumption that $u\in L^1(E;\nu)$.
\end{remark}

We close this section with the result  saying that  $N_\nu$ is the set of all possible zeros of positive non-trivial solutions to (\ref{eq5.2}).

\begin{theorem}
Assume that $(T_t)_{t\ge 0}$ is irreducible and  $\nu$ is a positive smooth measure. Then
\[
N_\nu= \bigcup_{u\in \mathcal A} \{u=0\},
\]
where
$\mathcal A =\{u\in L^1(E;m)\cap L^1(E;\nu):u\ge 0$, $u$ satisfies \mbox{\rm(\ref{eq5.2})}, $u$ is finely-continuous and $u\nequiv 0$\}.
\end{theorem}
\begin{proof}
Let $u\in\mathcal A$. Then  $u(x)>0$ for some $x\in E$.
Since $u$ is finely-continuous, there exists a
finely-open neighborhood $V_x$ of $x$ such that $u(y)>0$ for $y\in
V_x$. Since $V_x$ is finely-open,  $Cap_A(V_x)>0$ (see the comments on polar sets in Section \ref{sec2.2}), so
there exists $y_0\in E_\nu$ such that $u(y_0)>0$. Therefore, by
Theorem \ref{th5.1}, $\{u=0\}\subset N_\nu$. Consequently, $ \bigcup_{u\in
\mathcal A} \{u=0\}\subset N_\nu$. To prove the  opposite inclusion, we
use the assumption that $G^\nu(x_0,\cdot)\in L^1(E;m)$ for some
$x_0\in E_\nu$. In the proof of Proposition \ref{prop5.9wd4} we
have shown that $G^\nu(x_0,\cdot)\in L^1(E;\nu)$. By the very definition of the Green function, 
$G^\nu(x_0,\cdot)$ is  excessive  with respect to
$(P^\nu_t)_{t\ge 0}$. Let $v:= G^\nu(x_0,\cdot)\wedge 1$. Of course
$v\in L^1(E;m)\cap L^1(E;\nu)$ and $v$ is excessive with respect to
$(P^\nu_t)_{t\ge 0}$.  By \cite[Proposition 3.9]{GetoorGlover}, there exists a
positive Borel measure $\gamma$ such that $v=R^\nu\gamma$. Let
$\gamma_1:= R^\nu_1\gamma$.  Then $R^\nu\gamma_1\le v$. Let
$\eta$ be a strictly positive bounded function on $E$ such that
$R\eta$ is bounded (see \eqref{eq.acs5}), and let
\[
u(x):= \mathbb E_x\int_0^\zeta (\gamma_1\wedge \eta)(X_r) e^{-A^\nu_r}\,dr,\quad x\in E.
\]
Observe that $u=R^\nu(\gamma_1\wedge\eta)\le R^\nu\gamma_1\le v\in  L^1(E;m)\cap L^1(E;\nu)$ (see \eqref{eq4.s21}).
Since $u$ is excessive with respect to $(P^\nu_t)_{t\ge 0}$ (see \eqref{eq.acs0}), it satisfies (\ref{eq5.2}). 
By Proposition  \ref{prop4.1pw}(iii),  $u$ is finely-continuous. Thus $u\in \mathcal A$. It is clear that $\{u=0\}=N_\nu$, which proves that $N_\nu\subset \bigcup_{u\in \mathcal A} \{u=0\}$.
\end{proof}

\section{Strong maximum principle}
\label{sec6}

In this section we are concerned with the strong maximum principle (SMP)
for operators of the form $-A+V$, where $V:E\to [0,+\infty]$ is assumed to be merely a  Borel measurable function
(no additional assumptions). In consequence $\nu:=V\cdot m$ is not, in general, a smooth measure, and so
the results of the previous sections cannot be directly applied. 
Nevertheless, $N_{V\cdot m}$ is still well defined.  Moreover, 
\begin{equation}
\label{eq.nvvn}
N_{V\cdot m}=\{x\in E: P_x(\int_0^tV(X_r)\,dr=\infty,\, t>0)=1\}.
\end{equation}
The above equation one easily gets repeating the proof of \cite[Proposition 3.2]{K:NA1} with $\nu_n:= (V\wedge n)\cdot m$.

Fix a set $\mathcal C$ of {\em test functions} for $A$. 
We say that SMP holds for an operator $-A+V$ if  for any finely-continuous  positive function $u\in L^1_{loc}(E;m)\cap L^1(E;V\cdot m)$ 
that satisfies \eqref{eq5.1a} with $\nu=V\cdot m$, we have the following  implication: if $u(x_0)=0$ for some $x_0\in E$, then $u(x)=0,\, x\in E$.

\begin{theorem}
\label{th.iffff}
Let $V:E\to [0,+\infty]$ be a Borel measurable function.
Assume that   $(T_t)_{t\ge 0}$ is irreducible. Then the strong maximum principle holds for $-A+V$
if and only if 
\begin{enumerate}
\item[(1)] either $N_{V\cdot m}=\emptyset$
\item[(2)] or $N_{V\cdot m}\neq\emptyset$ and there is no non-trivial positive finely-continuous solution to \eqref{eq5.1a} with $\nu=V\cdot m$.
\end{enumerate}
\end{theorem}
\begin{proof}
Set $\nu:=V\cdot m$.
That  (1)  implies SMP follows from Theorem \ref{th5.1}. The implication (2)$\Rightarrow$SMP is trivial.
Now assume that SMP holds and $x_0\in N_\nu\neq\emptyset$. Let $D$ be an open relatively compact subset of $E$
such that $x_0\in D$. Let $f$ be a positive  finely-continuous solution to \eqref{eq5.1a}.
Repeating  the argument of the proof of Theorem \ref{prop5.1}  that led to \eqref{eqmm0},
we  conclude \eqref{eqmm0} with $\hat u$ replaced by $f$ and  $\int_0^t \hat u(X_r)\,dA^\nu_r$
replaced by $A^{f\cdot\nu}_t$. Observe that although $\nu$ is not smooth, we know that $f\cdot \nu$ is smooth, as it is a bounded measure,
thus, $A^{f\cdot\nu}$  is well defined. Now, repeating the argument of Step 3. of the proof of Theorem \ref{prop5.1}
that led to \eqref{eq5.5}, we find that 
\begin{align}
\label{eq5.4wf}
f_k(X_t)=f_k(X_0)+A^{f_k\cdot \nu}_t-A^{\beta_k}_t+L^k_t,\quad t\ge 0\,\,\,\mbox{q.a.s.},
\end{align}
for a local martingale $L^k$ and  a positive smooth  Green bounded measure $\beta_k$,
where $f_k:= f\wedge k$.
Let $\nu_l:=(V\wedge l)\cdot m$ (clearly it is a smooth measure), and let 
$\{\tau_n\}$  be a fundamental sequence for the local martingale $L^k$.
Applying integration by parts formula to $e^{-A^{\nu_l}_t}f_k(X_t)$ we find that 
for any stopping time $\alpha$ the following equation holds in $D$,
\begin{align*}
f_k(x)=\mathbb E_x\Big[e^{-A^{\nu_l}_{\alpha\wedge\tau_D\wedge\tau_n}}f_k(X_{\alpha\wedge\tau_D\wedge\tau_n})+\int_0^{\alpha\wedge\tau_D\wedge\tau_n}e^{-A^{\nu_l}_r}\,dA^{\beta_k}_r+A^{f_k\cdot\nu_l}_{\alpha\wedge\tau_D\wedge\tau_n}-A^{f_k\cdot\nu}_{\alpha\wedge\tau_D\wedge\tau_n}\Big]\quad\text{q.e.}
\end{align*}
Letting $l\to\infty$ and then $n\to\infty$ yields 
\begin{align}
\label{eq.letlet}
f_k(x)=\mathbb E_x\Big[e^{-A_{\alpha\wedge\tau_D}}f_k(X_{\alpha\wedge\tau_D})+\int_0^{\alpha\wedge\tau_D}e^{-A_r}\,dA^{\beta_k}_r\Big] \quad \text{q.e.  in  }D,
\end{align}
where $A_t:=\int_0^t V(X_r)\,dr,\, t\ge 0$. Let $w(x)$ denote the right-hand side of \eqref{eq.letlet} and let $w_l(x)$
denote the right-hand side of \eqref{eq.letlet} but with $A$ replaced by $A^{\nu_l}$. By Step 5. of the proof of Theorem \ref{prop5.1}
$w_l$ is finely-continuous. Clearly, $w=\inf_{l\ge 1}w_l$. Therefore, since  $f_k$ is assumed to by finely-continuous, we conclude
from \eqref{eq.letlet} that 
\begin{align}
\label{eq.letlet1}
f_k(x)\le \mathbb E_x\Big[e^{-A_{\alpha\wedge\tau_D}}f_k(X_{\alpha\wedge\tau_D})+\int_0^{\alpha\wedge\tau_D}e^{-A_r}\,dA^{\beta_k}_r\Big], \quad x\in D.
\end{align}
From this and  \eqref{eq.nvvn}, we deduce that $u(x_0)=0$. Thus, by SMP, $u(x)=0,\, x\in E$.
\end{proof}

\begin{remark}
\label{rem.bst}
The conclusion of Theorem \ref{th.iffff}, in case $A=\Delta$ and $d=1$, follows from \cite{BST},
where the authors  went a step further and give a necessary and sufficient condition on $V$ guaranteeing (2) of Theorem \ref{th.iffff}.
(recall here that in  case $A=\Delta$ and $d=1$, fine topology and Euclidean topology agree).
\end{remark}

\section{Schr\"odinger equations with $L^p$ potentials }
\label{sec7}

For  $p> 1$ and $A\subset E$, we set
\[
C_p(A)=\inf\{\|f\|^p_{L^p}: f\ge 0,\, f\in L^p(E;m),\, Rf\ge \mathbf{1}_A\}.
\]
We also let
\[
C_1(A)=\inf\{\|\mu\|_{TV}: \mu \in\MM^+(E),\, R\mu\ge \mathbf{1}_A\}
\]
for any $A\subset E$.
The above set functions are called  Riesz capacities.
For any $p\ge 1$ and $A\in \BB^n(E)$ we also set
\[
c_p(A)= \sup\{\|\mu\|_{TV}: \mu\in\MM^+_b(E),\, \mu(A^c)=0,\, \|R\mu\|_{L^{p'}}\le 1\}.
\]
Observe   that $c_1(A)\le C_1(A)$
for every $A\in\BB^n(E)$. Indeed, let $\mu\in\MM^+(E)$ be such that $\mu(A^c)=0$ and $R\mu\le 1$, and let $\nu\in\MM^+_b(E)$
be such that $R\nu\ge \mathbf{1}_{A}$. Then
\[
\|\mu\|_{TV}=\mu(A)\le \langle R\nu,\mu\rangle=\langle \nu,R\mu\rangle\le \|\nu\|_{TV},
\]
from which the desired result easily follows. By \cite[Exercise 2.2.2]{FOT},
for every compact $K\subset E$,
\[
Cap_A(K)=c_1(K).
\]
Since $Cap_A$ is a Choquet capacity, we conclude from the
above equality that
\begin{equation}
\label{eq6.1} Cap_A\ll c_1\le C_1.
\end{equation}

\begin{lemma}
\label{lm6.1}
If $f\in L^p(E;m)$ for some $p\ge 1$, then $f\cdot m$ is a smooth measure.
\end{lemma}
\begin{proof}
Let $g$ be a strictly positive function in $L^{p'}(E;m)$. Then $R_1g\in L^{p'}(E;m)$. Of course, $R_1 g$ is finely-continuous (see \eqref{eq.acs0}) and $\langle R_1 g,f\rangle<\infty$.
\end{proof}

\begin{theorem}
\label{th6.1} Let $\mathcal C$ be a set of test functions for $A$.  Let  $(T_t)_{t\ge 0}$ be irreducible and  $V$ be a
positive Borel function such that $V\in L^p(E;m)$ for some $p\ge
1$. Assume that $u\in L^1(E;m)\cap L^1(E;V\cdot m)$ is  a positive
function such that
\begin{equation}
\label{eq6.66}
\langle u,-A\eta\rangle+\langle u\cdot V,\eta\rangle\ge 0,\quad \eta\in \mathcal C.
\end{equation}
Then there exists an $m$-version $\check u$
of $u$ which is   finely-continuous on $E_\nu$. Moreover, 
if $C_p(\{\check u=0\})>0$, then $\check u\equiv0$.
\end{theorem}
\begin{proof}
Let $\check u$ be the function constructed in the proof of Theorem \ref{prop5.1}. By Lemma \ref{lm6.1}, $\nu:=V\cdot m$ is a positive smooth measure.
Since $N_{\nu}\subset \{RV=\infty\}$, it follows from  \cite[Theorem 3]{Meyers} (see also the comments at the beginning of \cite[Section 3]{Meyers}) that $C_p(N_{\nu})=0$. Hence,  since $C_p(\{\check u=0\})>0$,  there is  $x\in E_{\nu}$ such that $\check u(x)=0$. By Theorem \ref{th5.1}, $\check u\equiv 0$.
\end{proof}

\begin{corollary}
\label{cor5.2g}
Let $(T_t)_{t\ge 0}$ be irreducible, $V$ be a positive function such that $V\in L^1(E;m)$ and $u$  be a positive function in $L^1(E;m) \cap L^1(E;V\cdot m)$
such that \mbox{\rm(\ref{eq6.66})} is satisfied.
If $Cap_A(\{\check u=0\})>0$, then $\check u\equiv 0$.
\end{corollary}
\begin{proof}
Assume that $Cap_A(\{\check u=0\})>0$. Then, by
(\ref{eq6.1}), $C_1(\{\check u=0\})>0$, so  by Theorem
\ref{th6.1}, $\check u\equiv 0$.
\end{proof}

\section{The set of test functions and finely-continuous versions of supersolutions}
\label{sec8}
In the present section we are concerned with the structure of 
the set  of test functions $\mathcal C$ in \eqref{eq5.1a}.  Usually, in applications, we are interested in cases when
$\mathcal C= C_c^\infty(\BR^d)$ or $\mathcal C=C_b^\infty(\BR^d)$ are allowed. In the second part of the section
we will provide several remarks and simple results on finely-continuous versions of functions.
Especially, we focus on the Laplacian and the fractional Laplacian.
In these two particular cases one may provide  some special formulas for finely-continuous $m$-versions $\check u$.
We shall close this section with a result on equivalence between    Riesz's capacity for the fractional Laplacian and some    Sobolev capacity.

\subsection{The set of test functions}
We will start with the following useful result.

\begin{proposition}
\label{prop.sc1}
Assume that $u\in\mathfrak D_e(\EE)\cap L^1(E;\nu)$ and $\mathcal C(\EE)$ is a standard core for $\EE$:
 $\mathcal C(\EE)\subset  \mathfrak D_e(\EE)\cap C_c(E)$ is dense in $C_c(E)$ and $\mathfrak D_e(\EE)$,
 with standard norms, and $u\in \mathcal C(\EE)$ implies $\phi(u)\in\mathcal C(\EE)$ for any $1$-Lipschitz smooth function on $\BR$
 with $\phi(0)=0$ (see \cite[page 6]{FOT}).
 If 
 \begin{equation}
 \label{eq.fvi1}
 \EE(u,\eta)+\int_E\tilde u\eta\,d\nu\ge 0,\quad \eta\in\mathcal C(\EE),
 \end{equation}
 then the conclusion of Theorem \ref{th5.1} holds true, namely,  
there exists an $m$-version $\check u$
of $u$ which is finely-continuous on $E_\nu$, and if $\check u(x)=0$ for some $x\in E\setminus N_\nu$, then $\check u\equiv0$.

\end{proposition}
\begin{proof}
First note, that $u$ being in $\mathfrak D(\EE)$ implies the existence of  $\tilde u$ (quasi-continuous $m$-version of $u$).
Since $\mathcal C(\mathcal E)$ is a standard core, we get at once that \eqref{eq.fvi1} holds for any $\eta\in\mathcal U_b$.
Clearly, for $\eta\in \mathcal U_b$, $ \EE(u,\eta)=\langle u,-A\eta\rangle$.
Therefore, the assumptions of Theorem  \eqref{th5.1} are satisfied.
\end{proof}
In many interesting cases a standard core of $\EE$ may be taken to be equal $C^\infty_c(\BR^d)$ (see e.g.  \cite{AR}, \cite[Example 1.2.4, Exercise 3.1.1]{FOT}, \cite{SU}).

Now let us return   to the case  we have no information about the regularity of $u$.

For positive functions $u\in\BB^n(E)$, we let
\[
H_D(u)(x)=\mathbb E_x u(X_{\tau_D}),\quad x\in E.
\]
By \cite[Theorem 4.3.2]{FOT}, $\Pi_D(u):= u-H_D(u)$, $\Pi_D: (\mathfrak D_e(\EE),\EE)\to (\mathfrak D_e(\EE_{|D}),\EE)$ is the orthogonal projection onto $\mathfrak D_e(\EE_{|D})$. Observe also that for any positive $u\in \mathfrak D_e(\EE_{|D})$,
\begin{equation}
\label{eq.exhar}
P^D_t(H_D(u))(x)=\mathbb E_x[(\mathbb E_{X_t}u(X_{\tau_D}))\mathbf1_{\{t<\tau_D\}}]
=\mathbb E_x[u(X_{\tau_D})\mathbf1_{\{t<\tau_D\}}]\le H_D(u)(x),\quad x\in D.
\end{equation}

\begin{lemma}
\label{lm7.7.7} 
Let $D$ be a relatively compact open subset of $E$.
Let     $u\in \mathfrak D(A)$ and  $\eta\in \mathfrak D(A_{|D})$ be  positive.
Then
\begin{equation}
\label{eq7.bpe1} (-A u,\eta)_{L^2(D;m)} \le
(u,-A_{|D}\eta)_{L^2(D;m)}.
\end{equation}
Moreover, there exists $c>0$ depending only on
$\|H_D(u)\|_{L^1(D;m)}$ such that
\begin{equation}
\label{eq7.bpe2}
|(u,-A_{|D}\eta)_{L^2(D;m)}-
(-A u,\eta)_{L^2(D;m)}| \le c\|\eta\|_\infty.
\end{equation}
\end{lemma}
\begin{proof}
Let $(T^D_t)_{t\ge 0}$ be the semigroup generated by $A_{|D}$.  By Dynkin's formula (see \cite[(4.4.2)]{FOT}), $u-H_D(u)\in \mathfrak D(A_{|D})$. Therefore,
\[
(-A u,\eta)_{L^2(D;m)}=\EE(u,\eta)=\EE(u-H_D(u),\eta)= \EE_{|D}(u-H_D(u),\eta)
=(u-H_D(u),-A_{|D} \eta)_{L^2(D;m)}.
\]
Consequently,
\begin{equation}
\label{eq.fhf1}
(-A u,\eta)_{L^2(D;m)}-(u, -A_{|D} \eta)_{L^2(D;m)}=-(H_D(u), -A_{|D} \eta)_{L^2(D;m)}.
\end{equation}
Next,
\begin{align}
\label{eq.fhf}
\nonumber
(H_D(u), -A_{|D} \eta)_{L^2(D;m)}&=\lim_{t\rightarrow0^+}\frac 1t(H_D(u), \eta-T^D_t\eta)_{L^2(D;m)}\\
&=\lim_{t\rightarrow0^+}\frac 1t(H_D(u)-T^D_t H_D(u), \eta)_{L^2(D;m)}.
\end{align}
Inequality  (\ref{eq7.bpe2}) follows from   (\ref{eq.fhf1}), (\ref{eq.fhf}) and \cite[Proposition 4.4]{BC}.
Combining \eqref{eq.exhar} with
(\ref{eq.fhf1}) and (\ref{eq.fhf})  yields (\ref{eq7.bpe1}).
\end{proof}

Let $j_\varepsilon$ be a standard mollifier on $\BR^d$.  In the remainder of this subsection we assume that
$E=\BR^d$, $C_c^\infty(E)\subset \mathfrak D(A)$ and
\begin{equation}
\label{eq.splot1}
A(u*j_\varepsilon)=j_\varepsilon* (Au),\quad u\in C_c^\infty(E).
\end{equation}
One easily shows that the above condition fulfill  L\'evy operators, i.e. operators of the form
\begin{equation}
\label{eq.levy}
Au(x):= a\Delta u(x)+\int_{\BR^d}\big(u(x+y)-u(x)-\nabla u(x)\cdot y\mathbf1_{B(0,1)}(y)\big)\,\mu(dy),
\end{equation}
where $a\ge 0$  and $\mu$ is a positive Borel measure on $\BR^d$ such that 
\[
\int_{\BR^d}\min\{1,|y|^2\}\,\mu(dy)<\infty.
\]

\begin{lemma}
\label{lm7.1} Let $D$ be a bounded open subset of $\BR^d$, and  $\nu$
be a positive smooth measure on $D$. Let $u\in L^1(D;m)\cap
L^1(D;\nu)$ be a positive quasi-continuous function  such that
\begin{equation}
\label{eq7.1}
\langle u,-A \xi\rangle +\langle u\cdot \nu,\xi\rangle\ge 0,\quad \xi\in C_c^\infty(D),\, \xi\ge 0.
\end{equation}
Then
\begin{equation}
\label{eq7.1orew}
\langle u,-A_{|D} \eta\rangle +\langle u\cdot \nu,\eta\rangle\ge 0,\quad \eta\in \mathcal U_b(A_{|D}).
\end{equation}
\end{lemma}
\begin{proof}
We extend $u$ (resp. $\nu$) to $\BR^d$ by letting $u=0$  (resp.
$\nu=0$) on $\BR^d\setminus D$, and then set $u_\varepsilon:=
j_\varepsilon*u$, $(u\cdot\nu)_\varepsilon:=
j_\varepsilon*(u\cdot\nu)$. Let $U$ be an  open set  such that $\overline
U\subset D$. Let $\eta\in \mathcal U_b(A_{|U})$. By
(\ref{eq7.1}), \eqref{eq.splot1}, for $\varepsilon,\delta>0$ small enough, we have
\begin{align*}
0&\le \langle u,-A (j_\varepsilon * (j_\delta*\eta)\rangle +\langle u\cdot \nu, (j_\varepsilon * (j_\delta*\eta)\rangle\\
&=\langle u_\varepsilon,-A (j_\delta*\eta)\rangle +\langle (u\cdot \nu)_\varepsilon, j_\delta*\eta\rangle
\\&  = \langle -A u_\varepsilon, j_\delta*\eta\rangle +\langle (u\cdot \nu)_\varepsilon, j_\delta*\eta\rangle.
\end{align*}
The most right term  converges to $ \langle -Au_\varepsilon,
\eta\rangle +\langle (u\cdot \nu)_\varepsilon, \eta\rangle$ as
$\delta\searrow 0$. From  this and (\ref{eq7.bpe1}), we conclude
\[
0\le \langle -Au_\varepsilon, \eta\rangle +\langle (u\cdot \nu)_\varepsilon, \eta\rangle\le
\langle u_\varepsilon, -A_{|U}  \eta\rangle +\langle (u\cdot \nu)_\varepsilon, \eta\rangle.
\]
Letting $\varepsilon \searrow 0$ we obtain
\[
\langle  u, -A_{|U}\eta\rangle +\langle u\cdot \nu, \eta\rangle\ge 0,\quad \eta\in \mathcal U_b(A_{|U}).
\]
By Step 1 of the proof of Theorem \ref{prop5.1}, $u+R^U(u\cdot\nu)$ is an excessive function with respect to $(P^U_t)$.
By  Riesz's decomposition theorem, there exists a positive  Borel
measure $\mu_U$ such that
\[
u+R^U(u\cdot \nu)=H_U(u)+R^U\mu_U.
\]
By the uniqueness argument, $\mu_U=(\mu_W)_{\lfloor U}$ for $U\subset W$ and $\bar W\subset D$.
Therefore there exists a positive Borel measure  $\mu$ on $D$ such that $\mu_{\lfloor U}=\mu_U$ for every open $U\subset D$ with $\bar U\subset D$. Thus
\[
u+R^U(u\cdot \nu)=H_U(u)+R^U\mu\quad \mbox{q.e.}
\]
Let $D_n$ be an increasing  sequence of  open sets  such that $\bar D_n\subset D$ and $\bigcup_{n\ge 1} D_n=D$. We have
\[
u+R^{D_n}(u\cdot\nu)=H_{D_n}(u)+R^{D_n}\mu\quad \mbox{q.e.}
\]
Letting $n\rightarrow \infty$ we get
\begin{equation}
\label{eq.iei2}
u+R^{D}(u\cdot\nu)=h+R^{D}\mu\quad \mbox{q.e.}
\end{equation}
with $h=\lim_{n\rightarrow \infty} H_{D_n}(u)$. One easily shows  that  $h$ is an excessive  function with respect to $(P^D_t)$.
Let $\eta\in \mathcal U_b (A_{|D})$. Then there exists $\rho \in\BB_b(E)$ such that $\eta=R^D\rho$. By (\ref{eq.iei2}),
\begin{align*}
\langle u,-A_{|D}\eta\rangle+\langle u\cdot\nu,\eta\rangle&=\langle u,\rho\rangle+\langle u\cdot\nu,R^D\rho\rangle=
\langle u,\rho\rangle+\langle R^D(u\cdot\nu),\rho\rangle\\
&=\langle h,\rho\rangle+\langle R^D\mu,\rho\rangle=\langle h,\rho\rangle+\langle \mu,R^D\rho\rangle\\
&=\langle h,-A_{|D}R^D\rho\rangle+\langle \mu,\eta\rangle\ge \langle h,-A_{|D}R^D\rho\rangle.
\end{align*}
This implies (\ref{eq7.1orew}) since $\langle h,-A_{|D}R^D\rho\rangle=\lim_{t\rightarrow0^+}\frac1t\langle h-T^D_th,R^D\rho\rangle\ge 0$.
\end{proof}

\begin{theorem}
\label{th7.1} Assume that   $(T^D_t)_{t\ge 0}$ is irreducible. Let $\nu$ be a positive smooth measure on $D$ and
$u\in L^1(D;m)\cap L^1(D;\nu)$ be a positive quasi-continuous
function such that  
\begin{equation}
\label{eq7.1jsndjf}
\langle u,-A \xi\rangle +\langle u\cdot \nu,\xi\rangle\ge 0,\quad \xi\in C_c^\infty(D),\, \xi\ge 0.
\end{equation}
Then there exists an $m$-version $\check u$ of $u$, which is  finely-continuous on $E_\nu\cap D$.
Moreover, if $\check u(x)=0$ for some $x\in E_\nu\cap D$, then $\check u\equiv0$ in $D$.
\end{theorem}
\begin{proof}
It  follows from Theorem \ref{th5.1} and Lemma \ref{lm7.1}.

\end{proof}

\begin{remark}
If $A$ is given by \eqref{eq.levy}, then, by \cite[Proposition 2.2]{Grzywny}, $(T^D_t)$ is irreducible provided
one of the following conditions holds:
\begin{enumerate}
\item[(1)] $a>0$ and $D$ is connected,
\item[(2)] $\overline{D}$ is a subset of the support of $\mu$.
\end{enumerate} 
\end{remark}

\begin{corollary}
Under the assumptions of Theorem \ref{th7.1},  if
$Cap_A(\{u=0\}\cap D)>0$, then $u=0$ q.e. in $D$.
\end{corollary}
\begin{proof}
By Theorem \ref{th7.1}, there exists an $m$-version  $\check u$ of
$u$, which is finely-continuous on $E_\nu\cap D$. Since $u$
is quasi-continuous, $\check u= u$ q.e. in $D$, so by the assumptions of
the corollary, $Cap_A(\{\check u=0\}\cap D)>0$. Since
$Cap_A(N_\nu)=0$, there exists $x\in E_\nu\cap D$ such that $\check
u(x)=0$. Hence, by Theorem \ref{th7.1}, $\check u=0$ in $D$, which implies that
$u=0$ q.e. in $D$.
\end{proof}

\subsection{Finely-continuous versions of supersolutions}

In general, if $u$ is a finely-continuous positive function on a finely-open
set $U\subset E$, then by \eqref{eq.acs2}
\begin{equation}
\label{pfor2}
u(x)=\lim_{r\searrow 0}\mathbb E_xu(X_{\tau_{B(x,r)}})=\lim_{t\searrow 0}\mathbb E_xu(X_t)=\lim_{\alpha\to\infty} \alpha\mathbb E_x\int_0^\infty e^{-\alpha t}u(X_t)\,dt,
\quad x\in U. 
\end{equation}
The last two equalities combined with \eqref{Green} imply  that if a function $u:E\to\BR$ has an $m$-version $\check u$
which is   finely-continuous on $U$, then  \eqref{fcw101} holds. The first equation in \eqref{pfor2} can be  recast  as follows
\begin{equation}
\label{pfor1}
u(x)=\lim_{r\searrow 0} \int_{E\setminus B(x,r)} u(y)\,P_{B(x,r)}(x,dy),
\end{equation}
where for fixed $x\in U$, $P_{B(x,r)}(x,dy)$ is a Borel measure on $E\setminus B(x,r)$, so called harmonic measure. 
The kernel  $P_{B(x,r)}(x,dy)$ is called in the literature the Poisson kernel. It is given by the following formula
\begin{equation}
\label{pfor}
P_{B(x,r)}(x,dy):= P_x(X_{\tau_{B(x,r)}}\in dy).
\end{equation}

In case of the Laplacian and the fractional Laplacian,  Poisson's kernels may be computed explicitly. 
In many cases, although no explicit formula is  known, asymptotic behavior of the Poisson kernel is well studied 
(see e.g. \cite{KK}).

\begin{lemma}
\label{lm7.fcv} Let $A=\Delta_{|D}$. Let $u$ be a finely-continuous positive bounded
function on $D$. Then
\[
u(x)=\lim_{r\rightarrow0^+} \dashint_{B(x,r)} u(y)\,dy,\quad x\in D.
\]
\end{lemma}
\begin{proof}
Let $a_{d,r}= m(B(x,r))=r^d \cdot a_d$ where
$a_d=\pi^{d/2}/\Gamma(\frac d2+1)$, and let $b_{d,r}= S(\partial
B(x,r))=r^{d-1}\cdot b_d$ where $b_d=2\pi^{d/2}/\Gamma(\frac d2)$.
Note that $b_d/a_d=d$. By using \cite[Proposition 1.21]{ChungZhao}
we find that
\[
\dashint_{B(x,r)} u(y)\,dy=\frac{1}{a_{d,r}}\int_0^r \int_{\partial B(x,s)} u(y)\,dS(y)\,ds
=\frac{1}{a_{d,r}}\int_0^r b_{d,s}\mathbb E_xu(X_{\tau_{B(x,s)}})\,ds.
\]
Hence
\[
\dashint_{B(x,r)} u(y)\,dy
=\frac{b_d}{a_{d}}\frac{1}{d}\Big(\frac{1}{r^d}\int_0^r \mathbb E_xu(X_{\tau_{B(x,s)}})\,d(s^{d})\Big)
=\frac{1}{r^d}\int_0^r \mathbb E_xu(X_{\tau_{B(x,s)}})\,d(s^{d}).
\]
Since $u$ is finely-continuous,
$\mathbb E_xu(X_{\tau_{B(x,s)}})\rightarrow u(x)$ as $s\searrow 0$, which
combined with the above equation gives the desired result.
\end{proof}

Now Theorem \ref{th7.1} can be restated as follows.

\begin{theorem}
\label{th7.1rest}
Let $D$ be an open subset of $\BR^d$ such that $(T^D_t)$ is irreducible.
Let $\nu$ be a positive smooth measure on $D$
and $u\in L^1(D;m)\cap L^1(D;\nu)$ be a positive quasi-continuous
function  such that \mbox{\rm(\ref{eq7.1jsndjf})} is satisfied.
If 
\[
\limsup_{r\rightarrow0^+}\dashint_{B(x,r)} u(y)\,dy=0
\]
for some $x\in E_\nu\cap D$, then $u=0$ q.e. in $D$.
\end{theorem}
\begin{proof}
Follows from Corollary \ref{wn.red}, Theorem \ref{th7.1}  and
Lemma \ref{lm7.fcv}.
\end{proof}

We shall provide  one another corollary, which  was the
main result of the recent paper by Orsina and Ponce
\cite{OrsinaPonce}.

For any compact $K\subset D$ and $p>1$, we define
\[
Cap_{W^{2,p}}(K):=\inf\{\|u\|_{W^{2,p}}: u\in
C_c^\infty(\BR^d),\, u\ge \mathbf{1}_{K}\}.
\]
In the standard way $Cap_{W^{2,p}}$ can be extended to
arbitrary set $A\subset \BR^d$ (see \cite[Definition 2.2.4]{AH}).

\begin{remark}
\label{rem7.czl} By the Calder\'on-Zygmund $L^p$-theory,
$Cap_{W^{2,p}}$ is equivalent to $C_p$ for the operator
$A=\Delta_{|D}$ (see  Section \ref{sec6} for the definition of
$C_p$).
\end{remark}

\begin{theorem}
Let $D$ be an open subset of $\BR^d$ such that $(T^D_t)$ is irreducible.
Let $p>1$ and $V\in L^p(D;m)$ be  positive.
Let $u\in L^1(D;m)\cap L^1(D;V\cdot m)$ be a positive function   that satisfies  
\[
\langle u,-\Delta \xi\rangle +\langle V u,\xi\rangle\ge 0,\quad \xi\in C_c^\infty(D),\,\xi\ge 0.
\]
Write
\[
Z:= \Big\{x\in D: \limsup_{r\rightarrow0^+}\dashint_{B(x,r)} u(y)\,dy=0\Big\}.
\]
If $\mbox{\rm Cap}_{W^{2,p}}(Z)>0$, then $u=0$ $m$-a.e. in $D$.
\end{theorem}
\begin{proof}
By Remark \ref{rem7.czl}, $C_p(Z)>0$. Write $\nu=V\cdot m$. Since $C_p(N_{\nu})=0$ (see the proof of Theorem \ref{th6.1}), $E_{\nu}\cap Z\neq\emptyset$, so
by Theorem \ref{th7.1rest}, $u=0$ $m$-a.e.
\end{proof}

\begin{remark}
\label{rem.fcom}
Let $D$ be a bounded domain in $\BR^d$.
Let $Z$ be as in  \eqref{eq.zet}, and, for given positive $f\in L^\infty(D;m)$, let $\omega_f\in W^{1,2}_0(D)\cap L^\infty(D;m)\cap L^1(D;V\cdot m)$
be a  unique solution to \eqref{eq.op}. By \cite[Proposition 3.2(ii),Theorem 6.4]{K:NA1} $\omega_f$ has a finely-continuous $m$-version $\check \omega_f$
and 
\[
\check w_f(x)=\mathbb E_x\int_0^{\tau_D}e^{-\int_0^tV(X_r)\,dr}f(X_t)\,dt,\quad x\in D.
\]
We see that $\omega_f(x)=0$, whenever $x\in N_V$. On the other hand $\{\omega_1=0\}=N_V$.
Thereby, applying Lemma \ref{lm7.fcv} yields that $N_V=Z$.
\end{remark}

In what follows,
$c_{d,\alpha}=\pi^{1+d/2}\Gamma(d/2)\sin{\pi\alpha} $ and
\[
I^{(\alpha)}_ru(x)=c_{d,\alpha} \int_{\BR^d\setminus B(x,r)}
\frac{r^{\alpha}}{|y-x|^d(|y-x|^2-r^2)^\alpha}u(y)\,dy,\quad x\in D.
\]
\begin{lemma}
\label{lm7.fcvf}
Let $A=\Delta^\alpha$ for an $\alpha\in (0,1)$.
Let $u$ be a finely-continuous positive bounded function on $D$. Then
\[
u(x)=\lim_{r\rightarrow0^+} I^{(\alpha)}_ru(x),\quad x\in D.
\]
\end{lemma}
\begin{proof}
It is well known that
$\mathbb E_xu(X_{\tau_{B(x,r)}})=I^{(\alpha)}_ru(x)$, $x\in D$ (see, e.g.,
\cite[Section 4]{Kwasnicki}). On the other hand, since $u$ is
finely-continuous,  $\mathbb E_xu(X_{\tau_{B(x,r)}})\rightarrow u(x)$ as
$r\rightarrow0^+$.
\end{proof}

Similarly to the case of the Laplacian, we can now  restate Theorem \ref{th7.1} as follows.

\begin{theorem}
\label{th7.1b.rest}
Let $\nu$ be a positive smooth measure on $D$ and
$u\in L^1(D;m)\cap L^1(D;\nu)$ be a positive quasi-continuous function
such that \mbox{\rm (\ref{eq7.1jsndjf})} is satisfied.
If $\limsup_{r\rightarrow0^+}I^{(\alpha)}_ru(x)=0$ for some $x\in
E_\nu\cap D$, then $ u=0$ q.e. in $D$.
\end{theorem}
\begin{proof}
Follows from Corollary \ref{wn.red},  Theorem \ref{th7.1} and
Lemma \ref{lm7.fcvf}.
\end{proof}

For $0<s<1$ and $p>1$, we define
\[
|u|_{W^{s,p}(D)}=\Big( \int_D\int_D\frac{|u(x)-u(y)|^p}{|x-y|^{d+sp}}\,dx\,dy\Big)^{1/p},
\]
and for an arbitrary $s>0$ such that $s\notin \mathbb N$ we define
\[
W^{s,p}(D)=\Big\{u\in W^{[s],p}(D): \Big|\frac{\partial^k u}{\partial x_1^{k_1}\dots\partial x_d^{k_d}}\Big|_{W^{s-[s],p}(D)}<\infty,\, |k|=[s],\, k\in\mathbb N^d\Big\},
\]
where  $|k|=k_1+\ldots+k_d$. We adopt the convention that
$W^{0,p}(D)=L^p(D)$.  Let $K$ be a compact subset of $D$. We
set
\[
Cap_{W^{2\alpha,p}}(K):=\inf\{\|u\|_{W^{2\alpha,p}(D)}:
u\in C_c^\infty(D),\, u\ge \mathbf{1}_{K}\}.
\]
In the standard way $Cap_{W^{2\alpha,p}}$ can be extended
to arbitrary set $A\subset D$ (see \cite[Definition 2.2.4]{AH}).

\begin{proposition}
\label{prop8.1} Let $K$ be a compact subset of $D$ and $p\ge 2$. Let $C_p$
be  Riesz's capacity defined for the operator $A=(\Delta^\alpha)_{|D}$.
Then $C_p(K)=0$ if and only if $\mbox{\rm Cap}_{W^{2\alpha,p}}(K)=0$.
\end{proposition}
\begin{proof}
Since $\{u\in C_c^\infty(D):u\ge \mathbf{1}_{K}\}\subset \{R^Df:
f\in L^p(D),\, R^Df\ge\mathbf{1}_K\}$,  we have
$C_p(K)\le Cap_{W^{2\alpha,p}}(K)$. To show the necessity
part, suppose that that $C_p(K)=0$.  By the definition of $C_p$,
for every $\varepsilon>0$ there exists a positive
$f_\varepsilon\in L^p(D)$ such that $R^Df_\varepsilon\ge
2\mathbf{1}_K$ and $\|f_\varepsilon\|_{L^p(D)}\le\varepsilon$. We
set $u= R^Df_\varepsilon$ and extend to $\BR^d$ by putting
$f_\varepsilon =0$ on $\BR^d\setminus D$. By  Dynkin's formula,
\[
u(x)+\mathbb E_xRf_\varepsilon(X_{\tau_D})=Rf_\varepsilon(x),\quad x\in D.
\]
Set $v(x)= Rf_\varepsilon(x)$, $h(x)=
\mathbb E_xRf_\varepsilon(X_{\tau_D})$, $x\in \BR^d$. By the
Calder\'on-Zygmund $L^p$-theory,
\[
\varepsilon\ge \|f_\varepsilon\|_{L^p(D)}
=\|f_\varepsilon\|_{L^p(\BR^d)}\ge c \|v\|_{W^{2\alpha,p}(\BR^d)}.
\]
Let $v_\delta=j_\delta*v$, where $j_\delta$ is the standard
mollifier. A  straightforward computation shows that
$\|v_\delta\|_{W^{2\alpha,p}(\BR^d)}\le
\|v\|_{W^{2\alpha,p}(\BR^d)}$. Hence $\varepsilon\ge
c\|v_\delta\|_{W^{2\alpha,p}(\BR^d)}$. Let $\xi\in C_c^\infty(D)$
be such that $\xi\ge\mathbf{1}_K$. By \cite[Theorem
1.4.1.1]{Grisvard}, there exists $c_\xi$ such that
\[
\|\xi v_\delta\|_{W^{2\alpha,p}(\BR^d)}
\le c_\xi\|v_\delta\|_{W^{2\alpha,p}(\BR^d)}.
\]
Since $R$ is strongly Feller, $v$ is l.s.c. Therefore for a
sufficiently small $\delta>0$, $v_\delta\ge \mathbf{1}_K$. Of
course $\xi v_\delta\ge \mathbf{1}_K$, $\xi v_\delta\in
C_c^\infty(D)$ and
\[
\varepsilon\ge c c_\xi \|\xi v_\delta\|_{W^{2\alpha,p}(\BR^d)}
\ge c c_\xi \|\xi v_\delta\|_{W^{2\alpha,p}(D)}.
\]
Since $\varepsilon>0$ was arbitrary, this implies that $\mbox{\rm
Cap}_{W^{2\alpha,p}}(K)=0$.
\end{proof}

\begin{remark}
The assertion of Proposition \ref{prop8.1} holds true for  $p\in
(1,2)$  and $\alpha=1/2$ (the proof is analogous to the proof
given above). In case  $p\in (1,2)$ and $\alpha\neq 1/2$ it is not
true that $\|f\|_{L^p(\BR^d)}\sim \|Rf\|_{W^{2\alpha,p}(\BR^d)}$
for $f\in L^p(\BR^d)$. As a consequence, in that case Proposition
\ref{prop8.1} does not hold as stated. However, it holds true if
in its formulation we replace the space $W^{2\alpha,p}(\BR^d)$ by
the Besov space $B^{2\alpha}_{p,2}(\BR^d)$ (see \cite[Theorem 5,
page 155]{Stein}).
\end{remark}

\begin{theorem}
Let $V$ be a positive function in $L^p(D;m)$ for some $p\ge 2$ and
$u\in L^1(D;m)\cap L^1(E;V\cdot m)$ be a positive function such
that 
\[
\langle u,-\Delta^\alpha \xi\rangle +\langle u\cdot V,\xi\rangle\ge0,\quad \xi\in C_c^\infty(D),\, \xi\ge 0.
\]
Let
$Z=\{x\in D:\limsup_{r\rightarrow0^+}I^{(\alpha)}_r(u)(x)=0\}$.
If  $\mbox{\rm Cap}_{W^{2\alpha,p}}(Z)>0$, then $u=0$ $m$-a.e.
\end{theorem}
\begin{proof}
By Proposition \ref{prop8.1}, $C_p(Z)>0$. Let $\nu=V\cdot m$.
Since $C_p(N_{\nu})=0$ (see the proof of Theorem \ref{th6.1}),
$E_{\nu}\cap Z\neq\emptyset$. Hence, by Theorem \ref{th7.1b.rest},
$u=0$ $m$-a.e.
\end{proof}

\subsection*{Acknowledgements}
{\small This work was supported by Polish National Science Centre
(Grant No. 2017/25/B/ST1/00878).}


\begin{thebibliography}{30}

\bibitem{AH}
Adams, D.R., Hedberg, L.I.: {\em Function Spaces and Potential Theory.}
Springer-Verlag, Berlin, 1996.

\bibitem{ABR}
Albeverio, S.,  Brasche, J.,  R\"ockner, M.: Dirichlet forms and generalized Schr\"odinger
operators. {\em Lecture Notes in Physics} {\bf 345}, Springer, Berlin, 1989, pp. 1--42.



\bibitem{AM1}
Albeverio, S.,  Ma, Z.:
Additive functionals, nowhere Radon and Kato class smooth measures associated with Dirichlet forms. {\em Osaka J. Math.} {\bf 29} (1992) 247--265.

\bibitem{AR}
Albeverio, S.,  R\"ockner, M.: Classical Dirichlet forms on topological vector spaces - closability and a Cameron-Martin formula.  {\em J. Funct. Anal.} {\bf 88} (1990) 395--436.

\bibitem{Ancona}
Ancona, A.: Une propri\'et\'e d'invariance des ensembles absorbants par perturbation d'un op\'erateur elliptique. {\em Commun. Partial Differ. Equ.} {\bf 4} (1979) 321--337.

\bibitem{BarriosMedina}
Barrios, B.,  Medina, M.: Strong maximum principles for fractional elliptic
and parabolic problems with mixed boundary conditions. {\em Proceedings of the Royal Society of Edinburgh}. DOI:10.1017/prm.2018.77.

\bibitem{BST}
Bertsch, M., Smarrazzo, F.,  Tesei, A.:
A note on the strong maximum principle. {\em J. Differential Equations} {\bf 259} (2015) 4356--4375.


\bibitem{BB}
B\'enilan, P., Brezis, H.: Nonlinear problems related to the
Thomas-Fermi equation. {\em J. Evol. Equ.} {\bf 3} (2004) 673--770.

\bibitem{BC}
Beznea, L., C\^impean, I.: Quasimartingales associated to Markov processes.
{\em Trans. Amer. Math. Soc.} {\bf 370} (2018) 7761--7787.

\bibitem{BCR}
Beznea, L., C\^impean, I., R\"ockner, M.:
Irreducible recurrence, ergodicity, and extremality of invariant measures for resolvents.
{\em Stochastic Process. Appl.} {\bf 128} (2018) 1405--1437.


\bibitem{Bil} Billingsley, P., Convergence of probability measures. Second edition. Wiley Series in Probability and Statistics: Probability and Statistics. 
A Wiley-Interscience Publication. John Wiley \& Sons, Inc., New York, 1999.

\bibitem{BL}
Biswas, A.,  Lorinczi, J.: Hopf’s lemma for viscosity solutions to a class of non-local equations with applications. 
{\em Nonlinear Anal.} {\bf 204} (2021) Paper No. 112194.

\bibitem{BG}
Blumenthal, M.R., Getoor, R.K.: {\em Markov Processes and
Potential Theory}. Academic Press, New York and London, 1968.

\bibitem{BV}
Bucur, C.,  Valdinoci, E.:  Nonlocal Diffusion and Applications.{\em  Lecture Notes of the Unione Matematica Italiana.}   {\bf  20}  (2016) Springer, Cham.

\bibitem{BP}
Brezis, H., Ponce, A. C.: Remarks on the strong maximum principle.
{\em Differential Integral Equations} {\bf 16} (2003) 1--12.


\bibitem{Calabi}
Calabi, E.:  An extension of E. Hopf's maximum principle with an application to
Riemannian geometry. {\em Duke Math. J.} {\bf 25} (1958)  45--56.



\bibitem{CW}
Chung, K. L., Walsh, J. B.: Meyer's theorem on predictability. {\em Z. Wahrscheinlichkeitstheorie und Verw. Gebiete} {\bf 29} (1974) 253--256.


\bibitem{CJPS}
Cinlar, E., Jacod, J., Protter, P.,  Sharpe, M.J.:
Semimartingales and Markov Processes. Z.
Wahrscheinlichkeitstheorie verw. Gebiete {\bf 54}, 161--219 (1980)

\bibitem{ChungZhao}
Chung, K.L.,  Zhao, Z.: {\em From Brownian Motion to Schr\"odinger's Equation},
Springer-Verlag, Berlin, 1995.

\bibitem{Juan}
D\'avila, J.: A strong maximum principle for the Laplace equation with
mixed boundary condition.
{\em J. Funct. Anal.} {\bf 183} (2001)  231--244.

\bibitem{FOT}
Fukushima, M., Oshima, Y., Takeda, M.: {\em Dirichlet forms and
symmetric Markov processes.  Second revised and extended edition.}
Walter de Gruyter, Berlin, 2011.

\bibitem{FS}
Fukushima, M., Shima, T.: On a spectral analysis for the Sierpi\'nski gasket. {\em Potential Anal.} {\bf 1} (1992) 1--35.

\bibitem{Getoor}
 Getoor, R.K.: Transience and recurrence of Markov processes, in: S\'eminaire De Probabilit\'es (Strasbourg) XIV,
in: {\em Lecture Notes in Math.} {\bf  784} Springer (1980)  397--409.

\bibitem{GetoorGlover}
Getoor, R.K., Glover, J.: Riesz decomposition in Markov process
theory. {\em Trans. Amer. Math. Soc.} {\bf 285} (1984) 107--132.

\bibitem{Grisvard}
Grisvard, P.: {\em Elliptic problems in nonsmooth domains}.
Monographs and Studies in Mathematics, {\bf 24}. Pitman (Advanced Publishing Program), Boston, MA, 1985. 

\bibitem{Grzywny}
Grzywny, T.:  Intrisic ultracontractivity for L\'evy processes.
{\em  Probab. Math. Statist.} {\bf 28}  (2008) 91--106.

\bibitem{GT}
Gilbarg, D.; Trudinger, N. S. {\em Elliptic partial differential equations of second order.} 
Reprint of the 1998 edition. Classics in Mathematics. Springer-Verlag, Berlin, 2001.

\bibitem{HN}
Hansen, W.,  Netuka, I.: Semipolar Sets and Intrinsic Hausdorff Measure
{\em Potential Analysis} {\bf 51} (2019) 49--69.

\bibitem{Hopf}
Hopf, E.:  {\em Elementare Bemerkungen \"uber die L\"osungen partieller Differentialgleichungen
zweiter Ordnung vom elliptischen Typus.}, Sitzungsberichte, Preussische Akademie der
Wissenschaften, {\bf 19} (1927) 147--152.

\bibitem{KK}
Kang, J.,  Kim, P.:  On estimates of Poisson kernels for symmetric L\'evy processes. 
{\em J. Korean Math. Soc. } {\bf 50} (2013) 1009--1031





\bibitem{K:NA1}
Klimsiak, T.:  Schr\"odinger equations with smooth measure potential and general measure data. 
{\em Nonlinear Anal.} {\bf 218} (2022) Paper No. 112774

\bibitem{Kwasnicki}
Kwa\'snicki, M.: Fractional Laplace operator and its properties.
{\em Handbook of fractional calculus with applications. Vol. 1}.
De Gruyter, Berlin, 2019, pp. 159--193.

\bibitem{Littman}
Littman, W.: A strong maximum principle for weakly $L$-subharmonic functions. {\em  J. Math. Mech.}  {\bf 8} (1959)  761--770.

\bibitem{MR}
Ma, Z.-M., R\"ockner, M.: {\em Introduction to the Theory of
(Non-Symmetric) Dirichlet Forms}. Springer, Berlin, 1992.





\bibitem{Meyers}
Meyers, N.G.: A Theory of Capacities for Potentials of Functions
in Lebesgue Classes. {\em Math. Scand.} {\bf 26} (1970)  255--292.


\bibitem{OP1}
Orsina, L., Ponce, A.C.: Semilinear elliptic equations and
systems with diffuse measures. {\em J. Evol. Equ}. {\bf 8} (2008)
781--812.

\bibitem{OrsinaPonce}
Orsina, L.,  Ponce, A.C.: Strong maximum principle for
Schr\"odinger operators  with singular potential. {\em Ann. Inst.
H. Poincar\'e Anal. Non Lin\'eaire} {\bf 33} (2016) 477--493.

\bibitem{OP}
Orsina, L., Ponce, A.C.: On the nonexistence of Green's function and failure
of the strong maximum principle.  {\em J. Math. Pures Appl.} {\bf 134} (2020) 72--121

\bibitem{Oshima}
Oshima, Y.: {\em Semi-Dirichlet forms and Markov processes.} De Gruyter Studies in Mathematics {\bf 48} Berlin, 2013. x+284 pp.


\bibitem{Pazy}
Pazy, A.:  {\em Semigroups of linear operators  and applications
to partial differential equations.}
Springer-Verlag, New York, 1983.




\bibitem{Protter}
Protter, P.: {\em Stochastic Integration and Differential
Equations}. Second Edition. Springer, Berlin, 2004.

\bibitem{Ramaswamy}
Ramaswamy, S.: Fine connectedness and $\alpha$-excessive functions. {\em Ann. Inst. Fourier} (Grenoble) {\bf 22} (1972) 165--168.

\bibitem{SU}
Schilling, R. L.  Uemura, T.:
On the structure of the domain of a symmetric jump-type Dirichlet form.
{\em Publ. Res. Inst. Math. Sci. } {\bf 48} (2012) 1--20.

\bibitem{Stampacchia}
Stampacchia, G.:  Le probl\`eme de Dirichlet pour les \`equations elliptiques du second ordre \`a
coefficients discontinus. {\em Annales de l'Institut Fourier, Grenoble,} {\bf 15} (1965) 189--258.

\bibitem{Stein}
Stein, E.: {\em Singular Integrals and Differentiability Properties of Functions}. 
Princeton University Press, Princeton, 1970.



\end{thebibliography}
\end{document}